\documentclass[oneside, a4paper,11pt,reqno]{amsart}

\usepackage[english]{babel}
\usepackage[colorlinks]{hyperref}
\usepackage{epsfig}
\usepackage{mathtools}
\usepackage{amsmath}

\usepackage{amsthm}

\usepackage{amssymb}
\usepackage{amscd}
\usepackage{tikz}
\usepackage{graphicx}
\usepackage{color}
\usepackage{verbatim}
\usepackage[utf8]{inputenc}
\usepackage{pgf,tikz}
\usepackage{mathrsfs}
\usetikzlibrary{arrows}

\newtheorem{thm}{Theorem}[section]
\newtheorem{lem}[thm]{Lemma}
\newtheorem{prop}[thm]{Proposition}
\newtheorem{fact}[thm]{Fact}

\newtheorem{hyp}[thm]{Hypothesis}

\newtheorem{cor}[thm]{Corollary}
\newtheorem{defn}[thm]{Definition}

\newtheorem{rem}[thm]{Remark}

\newcommand{\Z}{\mathbb{Z}}
\newcommand{\N}{\mathbb{N}}
\newcommand{\M}{\mathcal{M}}

\newcommand{\f}{\varphi}
\newcommand{\mup}{\overline{\mu}}
\newcommand{\tp}{\overline{T}}
\newcommand{\Mp}{\overline{M}}
\newcommand{\R}{\mathcal{R}}

\newcommand{\e}{\epsilon_0}

\newcommand{\Fb}{\overline F}

\title[Slow-fast systems in infinite measure, with or without averaging]{
Slow-fast systems in continuous time and infinite measure, with or without averaging
}
\address{Univ Brest,  CNRS  UMR 6205, Laboratoire de Mathématiques 
de Bretagne Atlantique.}
\author{Maxence Phalempin}

\begin{document}

\maketitle

\begin{abstract}
This paper studies the asymptotic behaviour of the solution of a differential equation perturbed by a fast flow preserving an infinite measure. This question is related with limit theorems for non-stationary Birkhoff integrals. 
We distinguish two settings with different behaviour: the integrable setting (no averaging phenomenon) and
the case of an additive "centered" perturbation term (averaging phenomenon). 
The paper is motivated by the case where the perturbation comes from the $\Z$-periodic Lorentz gas flow or from the geodesic flow over a $\Z$-cover of a negatively curved compact surface. 
We establish limit theorems in more general contexts. 
\end{abstract}

\section{introduction}
We call perturbed differential equation or equivalently slow fast system in continuous time the following Cauchy problem

\begin{align}\label{pertx}
\frac{dX^{\epsilon}_t(x_0,\omega)}{dt}=\tilde f(X^{\epsilon}_t(x_0,\omega),\f_{\frac t \epsilon }(\omega)) ,\;X_0^\epsilon=x_0\, ,\quad \forall t\in [0,S]
\end{align}
with $\tilde f :\mathbb{R}^d \times \M \rightarrow \mathbb{R}^d$ 
a bounded measurable function, uniformly Lipschitz in its first coordinate, 
and $(\f_s)_{s\geq 0}$ a flow preserving the measure $\mu$ over a measurable space $\M$. The term $\f_{t/\epsilon}$ represents the fast motion when $\epsilon \rightarrow 0^+$,
 whereas $X_t^\epsilon$ is considered as a slow motion. The behavior of $(X_t^\epsilon)_{t\geq 0}$ has been widely studied when the flow $\f_t$ preserves a probability measure and the resulting 
dynamical system is mixing, we can notably cite the work done for exponentially mixing flows and transformations as Anosov flows or  billiard transformations by Anosov \cite{anosovslowfast}, Arnold \cite{arnoldequadiff}, Khasminskii \cite{kasminski}, Kifer \cite{kifer} and Pène \cite{phdpene}, among others.
In these settings they proved that while the dynamics accelerates ($\epsilon \rightarrow 0^+$), the perturbed solution $(X_t^\epsilon)_{t\geq 0}$ converges to the averaged motion $(W_t)_{t\geq 0}$ solution of the following ordinary averaged differential equation
\begin{align}\label{equaw}
\frac{dW_t(x_0)}{dt}=\bar f(W_t(x_0))\, , \quad W_0(x_0)=x_0\, , \quad \forall t\in [0,S],
\end{align}
where $\bar f(x):=\int_{\M} \tilde f(x,.)d\nu$. They also identify (see \cite{kasminski}, \cite{phdpene}) the limit in distribution of the error term
\begin{equation}\label{error}
E_t^\epsilon(x_0,\omega):=X_t^\epsilon(x_0,\omega)-W_t(x_0)
\end{equation}
between the solution $W_t$ of the ordinary differential equation~\eqref{equaw} and the solution
$X_t^\epsilon$ of the perturbation~\eqref{pertx}
by $\varphi$ of equation~\eqref{equaw}. 
Other slow fast systems have been studied with fast motion belonging to less chaotic probability preserving dynamical system such as Pommeau Manneville dynamics  by Chevyrev, Friz, Korepanov and Melbourne in \cite{korepanovslowfast}.\\
Note that, in the particular case where $\tilde f(x,\omega)=\tilde f_0(\omega)$ does not depend on its first coordinate, then $X_t^\epsilon(x_0,\cdot)-x_0$ is simply 
the averaged Birkhoff integral $\epsilon\int_0^{t/\epsilon}f_0\circ \varphi_s\, ds$ and the error $E_t^\epsilon(x_0,\cdot)$ is then $\epsilon\int_0^{t/\epsilon}(\tilde f_0-\bar f)\circ \varphi_s\, ds$. Thus the study of $X_t^\epsilon$ and $E_t^\epsilon$ can be seen as a generalization of the study of Birkhoff integrals. 

Our motivation here is to study slow fast systems with the fast dynamics coming from a flow $(\varphi_t)_t$ preserving an infinite measure. A first study of perturbation by a infinite measure preserving dynamical systems has been done in~\cite{MP24discrete} in the discrete time context. In~\cite{MP24discrete}, we studied solutions of~\eqref{pertx} in which $\varphi_{t/\epsilon}$ was replaced by $T^{\lfloor t/\epsilon\rfloor}$ for some chaotic maps preserving an infinite measure
(such as the collision map in the $\mathbb Z$-periodic Lorentz gas).

We assume from now on that
$(\varphi_t)_t$ is a flow defined on some measurable space $\mathcal M$ and preserving an infinite ($\sigma$-finite) measure $\nu$. There are (at least) two natural settings to consider and that generalize naturally the above model from the probability measure preserving context to the infinite measure preserving context.

The first setting consists in considering the case where $\tilde f(x,\cdot)$ is integrable for all $x\in\mathbb R^d$. 
In this case, under natural assumptions, $(X_t^\epsilon(x_0,\cdot)-x_0)$ behaves mostly as $\int_{\mathcal M}\tilde f(x_0,\cdot)\, d\nu$  
as $\epsilon\rightarrow 0$ 
(see Theorem~\ref{thmexempleintegrable}, and Section~\ref{secintegrable} for a general result for flows, see also \cite{MPphd} for the discrete time dynamical counterpart with stronger assumptions and a control in some $L^1$-norm). In the study of this first setting, $W_t^\epsilon(x_0)$ (obtained by taking
this time only $\bar f(x_0)=\int_{\mathcal M}\tilde f(x_0,\cdot)\, d\nu$) does not play any role.

A second  natural setting, more difficult to study and to which most of the present paper is dedicated, is the case when  $\tilde f=f+\bar f$ is split into a 
integrable part $f$ and a drift $\bar f$ independent of the dynamics. In this situation the differential equation~\eqref{pertx} will appear as a perturbation
of the differential equation~\eqref{equaw} and we will establish a limit theorem for $E_t^\epsilon$ as in the context of chaotic probability preserving dynamical systems. 
A first study of the discrete time counterpart of this second setting has been done in~\cite{MP24discrete}. 
We will study this second setting 
for general family of models  $(\M,\nu,\varphi)$ (see Section~\ref{secZext0}) which includes the $\Z$-periodic Lorentz gas flow and also the geodesic flow on a $\mathbb Z$-cover of a compact negatively curved surface. 
A presentation of these two examples may be found in Section \ref{secmodels}. 
These two flows describe the behaviour of a point particle moving at unit velocity on a 
$\mathbb Z$-periodic surface $\mathcal R_0$ either with negative curvature (in the case of geodesic flow) or (in the case of the Lorentz gas flow) 
on the flat tube $\mathbb R\times\mathbb T$ 
deprived 
of a periodic set of "round" obstacles on which the particle is elastically reflected. The set of configurations $\mathcal M$ is then
the set of couples of position and unit velocity, that is the unit tangent bundle $\mathcal M=T^1\mathcal R_0$ (up to identifying pre- and post-collision vectors at a collision time) of the surface $\mathcal R_0$ in which the particle evolves. 
For both flows, $\varphi_t$ maps a couple position/unit-velocity to the couple position/unit-velocity after time $t$. These flows  preserve the Lebesgue (or Liouville) measure $\nu$. 
The Lorentz gas has been introduced by Lorentz in~\cite{Lorentz} to model the displacement of an electron in a weakly magnetic metal, the round obstacles modeling the atoms. 
An example of $\mathbb Z$-cover of negatively curved compact surface 
is the $\Z$-cover of the surface 
$\mathcal R\subset \mathbb R\times\mathbb T^2$ of equation $\cos(2\pi u)+\cos(2\pi v)+\cos(2\pi w)=0$. 

\begin{figure}[htbp!]
\definecolor{ffffff}{rgb}{1.,1.,1.}
\definecolor{qqwwzz}{rgb}{0.,0.4,0.6}
\definecolor{qqzzqq}{rgb}{0.,0.6,0.}
\definecolor{ccqqqq}{rgb}{0.8,0.,0.}
\begin{tikzpicture}[line cap=round,line join=round,>=triangle 45,x=0.38720631737471994cm,y=0.36753802264368196cm]
\clip(-14.417905401737666,-5.601309281127577) rectangle (13.508488657234448,5.486198798433184);
\draw [shift={(-5.,5.)},line width=2.pt]  plot[domain=-1.5707963267948966:0.,variable=\t]({1.*3.8*cos(\t r)+0.*3.8*sin(\t r)},{0.*3.8*cos(\t r)+1.*3.8*sin(\t r)});
\draw [shift={(-5.,-5.)},line width=2.pt]  plot[domain=0.:1.5707963267948966,variable=\t]({1.*3.8*cos(\t r)+0.*3.8*sin(\t r)},{0.*3.8*cos(\t r)+1.*3.8*sin(\t r)});
\draw [shift={(5.,-5.)},line width=2.pt]  plot[domain=1.5707963267948966:3.141592653589793,variable=\t]({1.*3.8*cos(\t r)+0.*3.8*sin(\t r)},{0.*3.8*cos(\t r)+1.*3.8*sin(\t r)});
\draw [shift={(5.,5.)},line width=2.pt]  plot[domain=3.141592653589793:4.71238898038469,variable=\t]({1.*3.8*cos(\t r)+0.*3.8*sin(\t r)},{0.*3.8*cos(\t r)+1.*3.8*sin(\t r)});
\draw [shift={(0.,6.748541905584902)},line width=2.pt,color=ccqqqq]  plot[domain=4.1109263962649605:5.313851564504419,variable=\t]({1.*2.120707145172685*cos(\t r)+0.*2.120707145172685*sin(\t r)},{0.*2.120707145172685*cos(\t r)+1.*2.120707145172685*sin(\t r)});
\draw [shift={(0.,3.1849587164306037)},line width=2.pt,color=ccqqqq]  plot[domain=0.9866282605050797:2.1549643930847138,variable=\t]({1.*2.1758618662638587*cos(\t r)+0.*2.1758618662638587*sin(\t r)},{0.*2.1758618662638587*cos(\t r)+1.*2.1758618662638587*sin(\t r)});
\draw [shift={(0.,-6.639510247152966)},line width=2.pt,color=ccqqqq]  plot[domain=0.93896334991882:2.202629303670973,variable=\t]({1.*2.0317465025242645*cos(\t r)+0.*2.0317465025242645*sin(\t r)},{0.*2.0317465025242645*cos(\t r)+1.*2.0317465025242645*sin(\t r)});
\draw [shift={(0.,-3.308013578821793)},line width=2.pt,color=ccqqqq]  plot[domain=4.095498169574483:5.329279791194896,variable=\t]({1.*2.074323516101439*cos(\t r)+0.*2.074323516101439*sin(\t r)},{0.*2.074323516101439*cos(\t r)+1.*2.074323516101439*sin(\t r)});
\draw [shift={(3.453995848431412,0.)},line width=2.pt,color=qqzzqq]  plot[domain=-0.6600559403422821:0.6600559403422824,variable=\t]({1.*1.9570714950321333*cos(\t r)+0.*1.9570714950321333*sin(\t r)},{0.*1.9570714950321333*cos(\t r)+1.*1.9570714950321333*sin(\t r)});
\draw [shift={(6.546004151568588,0.)},line width=2.pt,color=qqzzqq]  plot[domain=2.481536713247511:3.8016485939320757,variable=\t]({1.*1.9570714950321333*cos(\t r)+0.*1.9570714950321333*sin(\t r)},{0.*1.9570714950321333*cos(\t r)+1.*1.9570714950321333*sin(\t r)});
\draw [shift={(-6.201593228809644,0.)},line width=2.pt,color=qqzzqq]  plot[domain=-0.784734758555695:0.784734758555695,variable=\t]({1.*1.6981832314332825*cos(\t r)+0.*1.6981832314332825*sin(\t r)},{0.*1.6981832314332825*cos(\t r)+1.*1.6981832314332825*sin(\t r)});
\draw [shift={(-3.798406771190356,0.)},line width=2.pt,color=qqzzqq]  plot[domain=2.356857895034098:3.926327412145488,variable=\t]({1.*1.698183231433282*cos(\t r)+0.*1.698183231433282*sin(\t r)},{0.*1.698183231433282*cos(\t r)+1.*1.698183231433282*sin(\t r)});
\draw [shift={(1.537878564151371,3.4428554380431455)},line width=2.pt,color=ccqqqq]  plot[domain=2.8150550197602158:4.725722860250245,variable=\t]({1.*1.3207605113569272*cos(\t r)+0.*1.3207605113569272*sin(\t r)},{0.*1.3207605113569272*cos(\t r)+1.*1.3207605113569272*sin(\t r)});
\draw [line width=2.pt,color=qqzzqq] (0.2869088524857311,3.866509906749658)-- (0.049149734332857804,3.13162120677704);
\draw [line width=2.pt,color=qqzzqq] (0.049149734332857804,3.13162120677704)-- (0.04546628601518405,3.407879830602569);
\draw [line width=2.pt,color=qqzzqq] (0.049149734332857804,3.13162120677704)-- (0.20385456367515536,3.326843967613747);
\draw [line width=2.pt,color=qqzzqq] (1.5554889042995157,2.12221233550572)-- (2.326911741420443,2.183971966317168);
\draw [line width=2.pt,color=qqzzqq] (2.326911741420443,2.183971966317168)-- (2.1077316506737254,2.0922221608883107);
\draw [line width=2.pt,color=qqzzqq] (2.326911741420443,2.183971966317168)-- (2.0771483821974392,2.229846869031597);
\draw (-1.043929583900772,4.058986987332946) node[anchor=north west] {$x$};
\draw (1.847709579459238,1.5751430906007076) node[anchor=north west] {$\varphi_t(x)$};
\draw [line width=2.pt,color=qqwwzz] (-0.6584340669401274,-0.7749470911311843) ellipse (0.3686181449777437cm and 0.3498940436568077cm);
\draw [line width=2.pt,color=qqwwzz] (0.15009171636533897,0.20369590078153288) ellipse (0.3686181449777437cm and 0.3498940436568077cm);
\draw [shift={(0.15009171636533897,0.20369590078153288)},line width=2.pt,color=ffffff]  plot[domain=1.2532945355341802:1.4261206946553242,variable=\t]({1.*0.9519941396540091*cos(\t r)+0.*0.9519941396540091*sin(\t r)},{0.*0.9519941396540091*cos(\t r)+1.*0.9519941396540091*sin(\t r)});
\draw [shift={(0.15009171636533897,0.20369590078153288)},line width=2.pt,color=ffffff]  plot[domain=2.3416987030507435:2.525250187992933,variable=\t]({1.*0.951994139654009*cos(\t r)+0.*0.951994139654009*sin(\t r)},{0.*0.951994139654009*cos(\t r)+1.*0.951994139654009*sin(\t r)});
\draw [shift={(0.15009171636533897,0.20369590078153288)},line width=2.pt,color=ffffff]  plot[domain=3.09789858851906:3.279448970308834,variable=\t]({1.*0.9519941396540093*cos(\t r)+0.*0.9519941396540093*sin(\t r)},{0.*0.9519941396540093*cos(\t r)+1.*0.9519941396540093*sin(\t r)});
\draw [shift={(0.15009171636533897,0.20369590078153288)},line width=2.pt,color=ffffff]  plot[domain=3.8703838219144058:4.050715286329725,variable=\t]({1.*0.9519941396540093*cos(\t r)+0.*0.9519941396540093*sin(\t r)},{0.*0.9519941396540093*cos(\t r)+1.*0.9519941396540093*sin(\t r)});
\draw [shift={(0.15009171636533897,0.20369590078153288)},line width=2.pt,color=ffffff]  plot[domain=4.768408581387503:4.950891818015055,variable=\t]({1.*0.9519941396540093*cos(\t r)+0.*0.9519941396540093*sin(\t r)},{0.*0.9519941396540093*cos(\t r)+1.*0.9519941396540093*sin(\t r)});
\draw [shift={(0.15009171636533897,0.20369590078153288)},line width=2.pt,color=ffffff]  plot[domain=5.717248896782086:5.934395582020829,variable=\t]({1.*0.9519941396540091*cos(\t r)+0.*0.9519941396540091*sin(\t r)},{0.*0.9519941396540091*cos(\t r)+1.*0.9519941396540091*sin(\t r)});
\draw [shift={(0.15009171636533897,0.20369590078153288)},line width=2.pt,color=ffffff]  plot[domain=0.37257978132003017:0.536950688165803,variable=\t]({1.*0.951994139654009*cos(\t r)+0.*0.951994139654009*sin(\t r)},{0.*0.951994139654009*cos(\t r)+1.*0.951994139654009*sin(\t r)});
\draw [line width=2.pt,color=qqwwzz] (-0.6584340669401274,-0.7749470911311843) ellipse (0.3686181449777437cm and 0.3498940436568077cm);
\draw [shift={(5.,5.)},line width=2.pt]  plot[domain=3.141592653589793:4.71238898038469,variable=\t]({-1.*3.8*cos(\t r)+0.*3.8*sin(\t r)},{0.*3.8*cos(\t r)+1.*3.8*sin(\t r)});
\draw [shift={(5.,-5.)},line width=2.pt]  plot[domain=1.5707963267948966:3.141592653589793,variable=\t]({-1.*3.8*cos(\t r)+0.*3.8*sin(\t r)},{0.*3.8*cos(\t r)+1.*3.8*sin(\t r)});
\draw [shift={(10.,-6.6395102471529635)},line width=2.pt,color=ccqqqq]  plot[domain=0.93896334991882:2.202629303670973,variable=\t]({-1.*2.0317465025242645*cos(\t r)+0.*2.0317465025242645*sin(\t r)},{0.*2.0317465025242645*cos(\t r)+1.*2.0317465025242645*sin(\t r)});
\draw [shift={(10.,-3.3080135788217904)},line width=2.pt,color=ccqqqq]  plot[domain=4.095498169574483:5.329279791194896,variable=\t]({-1.*2.074323516101439*cos(\t r)+0.*2.074323516101439*sin(\t r)},{0.*2.074323516101439*cos(\t r)+1.*2.074323516101439*sin(\t r)});
\draw [shift={(15.,-5.)},line width=2.pt]  plot[domain=0.:1.5707963267948966,variable=\t]({-1.*3.8*cos(\t r)+0.*3.8*sin(\t r)},{0.*3.8*cos(\t r)+1.*3.8*sin(\t r)});
\draw [shift={(15.,5.)},line width=2.pt]  plot[domain=-1.5707963267948966:0.,variable=\t]({-1.*3.8*cos(\t r)+0.*3.8*sin(\t r)},{0.*3.8*cos(\t r)+1.*3.8*sin(\t r)});
\draw [shift={(10.,6.748541905584904)},line width=2.pt,color=ccqqqq]  plot[domain=4.1109263962649605:5.313851564504419,variable=\t]({-1.*2.120707145172685*cos(\t r)+0.*2.120707145172685*sin(\t r)},{0.*2.120707145172685*cos(\t r)+1.*2.120707145172685*sin(\t r)});
\draw [shift={(10.,3.1849587164306064)},line width=2.pt,color=ccqqqq]  plot[domain=0.9866282605050797:2.1549643930847138,variable=\t]({-1.*2.1758618662638587*cos(\t r)+0.*2.1758618662638587*sin(\t r)},{0.*2.1758618662638587*cos(\t r)+1.*2.1758618662638587*sin(\t r)});
\draw [line width=2.pt,color=qqwwzz] (9.341565933059872,-0.7749470911311815) ellipse (0.3686181449777437cm and 0.34989404365680776cm);
\draw [line width=2.pt,color=qqwwzz] (10.150091716365338,0.20369590078153565) ellipse (0.36861814497774353cm and 0.3498940436568076cm);
\draw [shift={(33.79840677119036,0.)},line width=2.pt,color=qqzzqq]  plot[domain=2.356857895034098:3.926327412145488,variable=\t]({-1.*1.698183231433282*cos(\t r)+0.*1.698183231433282*sin(\t r)},{0.*1.698183231433282*cos(\t r)+1.*1.698183231433282*sin(\t r)});
\draw [shift={(13.798406771190358,0.)},line width=2.pt,color=qqzzqq]  plot[domain=2.356857895034098:3.926327412145488,variable=\t]({-1.*1.698183231433282*cos(\t r)+0.*1.698183231433282*sin(\t r)},{0.*1.698183231433282*cos(\t r)+1.*1.698183231433282*sin(\t r)});
\draw [shift={(36.20159322880964,0.)},line width=2.pt,color=qqzzqq]  plot[domain=-0.784734758555695:0.784734758555695,variable=\t]({-1.*1.6981832314332825*cos(\t r)+0.*1.6981832314332825*sin(\t r)},{0.*1.6981832314332825*cos(\t r)+1.*1.6981832314332825*sin(\t r)});
\draw [shift={(16.201593228809646,0.)},line width=2.pt,color=qqzzqq]  plot[domain=-0.784734758555695:0.784734758555695,variable=\t]({-1.*1.6981832314332825*cos(\t r)+0.*1.6981832314332825*sin(\t r)},{0.*1.6981832314332825*cos(\t r)+1.*1.6981832314332825*sin(\t r)});
\draw [shift={(-5.,5.)},line width=2.pt]  plot[domain=-1.5707963267948966:0.,variable=\t]({-1.*3.8*cos(\t r)+0.*3.8*sin(\t r)},{0.*3.8*cos(\t r)+1.*3.8*sin(\t r)});
\draw [shift={(-5.,-5.)},line width=2.pt]  plot[domain=0.:1.5707963267948966,variable=\t]({-1.*3.8*cos(\t r)+0.*3.8*sin(\t r)},{0.*3.8*cos(\t r)+1.*3.8*sin(\t r)});
\draw [shift={(-10.,-6.639510247152965)},line width=2.pt,color=ccqqqq]  plot[domain=0.93896334991882:2.202629303670973,variable=\t]({-1.*2.0317465025242645*cos(\t r)+0.*2.0317465025242645*sin(\t r)},{0.*2.0317465025242645*cos(\t r)+1.*2.0317465025242645*sin(\t r)});
\draw [shift={(-10.,-3.3080135788217926)},line width=2.pt,color=ccqqqq]  plot[domain=4.095498169574483:5.329279791194896,variable=\t]({-1.*2.074323516101439*cos(\t r)+0.*2.074323516101439*sin(\t r)},{0.*2.074323516101439*cos(\t r)+1.*2.074323516101439*sin(\t r)});
\draw [shift={(-10.,6.7485419055849025)},line width=2.pt,color=ccqqqq]  plot[domain=4.1109263962649605:5.313851564504419,variable=\t]({-1.*2.120707145172685*cos(\t r)+0.*2.120707145172685*sin(\t r)},{0.*2.120707145172685*cos(\t r)+1.*2.120707145172685*sin(\t r)});
\draw [shift={(-10.,3.184958716430604)},line width=2.pt,color=ccqqqq]  plot[domain=0.9866282605050797:2.1549643930847138,variable=\t]({-1.*2.1758618662638587*cos(\t r)+0.*2.1758618662638587*sin(\t r)},{0.*2.1758618662638587*cos(\t r)+1.*2.1758618662638587*sin(\t r)});
\draw [shift={(-15.,5.)},line width=2.pt]  plot[domain=3.141592653589793:4.71238898038469,variable=\t]({-1.*3.8*cos(\t r)+0.*3.8*sin(\t r)},{0.*3.8*cos(\t r)+1.*3.8*sin(\t r)});
\draw [shift={(-15.,-5.)},line width=2.pt]  plot[domain=1.5707963267948966:3.141592653589793,variable=\t]({-1.*3.8*cos(\t r)+0.*3.8*sin(\t r)},{0.*3.8*cos(\t r)+1.*3.8*sin(\t r)});
\draw [shift={(-16.54600415156859,0.)},line width=2.pt,color=qqzzqq]  plot[domain=2.481536713247511:3.8016485939320757,variable=\t]({-1.*1.9570714950321333*cos(\t r)+0.*1.9570714950321333*sin(\t r)},{0.*1.9570714950321333*cos(\t r)+1.*1.9570714950321333*sin(\t r)});
\draw [shift={(-13.453995848431411,0.)},line width=2.pt,color=qqzzqq]  plot[domain=-0.6600559403422821:0.6600559403422824,variable=\t]({-1.*1.9570714950321333*cos(\t r)+0.*1.9570714950321333*sin(\t r)},{0.*1.9570714950321333*cos(\t r)+1.*1.9570714950321333*sin(\t r)});
\draw [line width=2.pt,color=qqwwzz] (-10.658434066940128,-0.7749470911311871) ellipse (0.36861814497774364cm and 0.34989404365680765cm);
\draw [line width=2.pt,color=qqwwzz] (-9.849908283634662,0.2036959007815301) ellipse (0.36861814497774364cm and 0.34989404365680765cm);
\draw [shift={(10.150091716365338,0.20369590078153565)},line width=2.pt,color=ffffff]  plot[domain=3.0810853260437403:3.2534719836399044,variable=\t]({1.*0.9519941396540086*cos(\t r)+0.*0.9519941396540086*sin(\t r)},{0.*0.9519941396540086*cos(\t r)+1.*0.9519941396540086*sin(\t r)});
\draw [shift={(10.150091716365338,0.20369590078153565)},line width=2.pt,color=ffffff]  plot[domain=4.797184521457376:4.909167490307397,variable=\t]({1.*0.9519941396540089*cos(\t r)+0.*0.9519941396540089*sin(\t r)},{0.*0.9519941396540089*cos(\t r)+1.*0.9519941396540089*sin(\t r)});
\draw [shift={(10.150091716365338,0.20369590078153565)},line width=2.pt,color=ffffff]  plot[domain=3.9274506309145725:4.032111176477976,variable=\t]({1.*0.9519941396540089*cos(\t r)+0.*0.9519941396540089*sin(\t r)},{0.*0.9519941396540089*cos(\t r)+1.*0.9519941396540089*sin(\t r)});
\draw [shift={(10.150091716365338,0.20369590078153565)},line width=2.pt,color=ffffff]  plot[domain=5.611771614766113:5.757602315325147,variable=\t]({1.*0.9519941396540088*cos(\t r)+0.*0.9519941396540088*sin(\t r)},{0.*0.9519941396540088*cos(\t r)+1.*0.9519941396540088*sin(\t r)});
\draw [shift={(10.150091716365338,0.20369590078153565)},line width=2.pt,color=ffffff]  plot[domain=0.06441334162510681:0.21436621656617055,variable=\t]({1.*0.951994139654008*cos(\t r)+0.*0.951994139654008*sin(\t r)},{0.*0.951994139654008*cos(\t r)+1.*0.951994139654008*sin(\t r)});
\draw [shift={(10.150091716365338,0.20369590078153565)},line width=2.pt,color=ffffff]  plot[domain=0.94481265163191:1.0646083663734986,variable=\t]({1.*0.9519941396540086*cos(\t r)+0.*0.9519941396540086*sin(\t r)},{0.*0.9519941396540086*cos(\t r)+1.*0.9519941396540086*sin(\t r)});
\draw [shift={(10.150091716365338,0.20369590078153565)},line width=2.pt,color=ffffff]  plot[domain=1.9839478026543094:2.098347056771017,variable=\t]({1.*0.9519941396540087*cos(\t r)+0.*0.9519941396540087*sin(\t r)},{0.*0.9519941396540087*cos(\t r)+1.*0.9519941396540087*sin(\t r)});
\draw [shift={(-9.849908283634662,0.2036959007815301)},line width=2.pt,color=ffffff]  plot[domain=3.125526756001007:3.233722789964543,variable=\t]({1.*0.951994139654009*cos(\t r)+0.*0.951994139654009*sin(\t r)},{0.*0.951994139654009*cos(\t r)+1.*0.951994139654009*sin(\t r)});
\draw [shift={(-9.849908283634662,0.2036959007815301)},line width=2.pt,color=ffffff]  plot[domain=4.810430807037666:4.948474356880295,variable=\t]({1.*0.951994139654009*cos(\t r)+0.*0.951994139654009*sin(\t r)},{0.*0.951994139654009*cos(\t r)+1.*0.951994139654009*sin(\t r)});
\draw [shift={(-9.849908283634662,0.2036959007815301)},line width=2.pt,color=ffffff]  plot[domain=3.8712089586447846:3.975435973019547,variable=\t]({1.*0.951994139654009*cos(\t r)+0.*0.951994139654009*sin(\t r)},{0.*0.951994139654009*cos(\t r)+1.*0.951994139654009*sin(\t r)});
\draw [shift={(-9.849908283634662,0.2036959007815301)},line width=2.pt,color=ffffff]  plot[domain=5.790147382783656:5.930373341510596,variable=\t]({1.*0.9519941396540088*cos(\t r)+0.*0.9519941396540088*sin(\t r)},{0.*0.9519941396540088*cos(\t r)+1.*0.9519941396540088*sin(\t r)});
\draw [shift={(-9.849908283634662,0.2036959007815301)},line width=2.pt,color=ffffff]  plot[domain=0.39862216062272726:0.510203524081383,variable=\t]({1.*0.9519941396540091*cos(\t r)+0.*0.9519941396540091*sin(\t r)},{0.*0.9519941396540091*cos(\t r)+1.*0.9519941396540091*sin(\t r)});
\draw [shift={(-9.849908283634662,0.2036959007815301)},line width=2.pt,color=ffffff]  plot[domain=1.2356784318523415:1.3337901550111841,variable=\t]({1.*0.951994139654009*cos(\t r)+0.*0.951994139654009*sin(\t r)},{0.*0.951994139654009*cos(\t r)+1.*0.951994139654009*sin(\t r)});
\draw [shift={(-9.849908283634662,0.2036959007815301)},line width=2.pt,color=ffffff]  plot[domain=2.2421594201295574:2.337733564551046,variable=\t]({1.*0.951994139654009*cos(\t r)+0.*0.951994139654009*sin(\t r)},{0.*0.951994139654009*cos(\t r)+1.*0.951994139654009*sin(\t r)});
\draw [line width=2.pt,color=qqwwzz] (9.341565933059872,-0.7749470911311815) ellipse (0.36861814497774376cm and 0.34989404365680776cm);
\draw [line width=2.pt,color=qqwwzz] (-10.658434066940128,-0.7749470911311871) ellipse (0.36861814497774353cm and 0.34989404365680754cm);
\draw [shift={(-6.201593228809644,0.)},line width=2.pt,color=ffffff]  plot[domain=5.868044438325866:5.9971915662813,variable=\t]({1.*1.6981832314332828*cos(\t r)+0.*1.6981832314332828*sin(\t r)},{0.*1.6981832314332828*cos(\t r)+1.*1.6981832314332828*sin(\t r)});
\draw [shift={(-6.201593228809644,0.)},line width=2.pt,color=ffffff]  plot[domain=0.05944576271585577:0.1752131699197991,variable=\t]({1.*1.6981832314332823*cos(\t r)+0.*1.6981832314332823*sin(\t r)},{0.*1.6981832314332823*cos(\t r)+1.*1.6981832314332823*sin(\t r)});
\draw [shift={(-6.201593228809644,0.)},line width=2.pt,color=ffffff]  plot[domain=0.4683433463231158:0.5633851181134774,variable=\t]({1.*1.6981832314332823*cos(\t r)+0.*1.6981832314332823*sin(\t r)},{0.*1.6981832314332823*cos(\t r)+1.*1.6981832314332823*sin(\t r)});
\draw [shift={(3.453995848431412,0.)},line width=2.pt,color=ffffff]  plot[domain=5.856857986581255:5.933997285500661,variable=\t]({1.*1.9570714950321337*cos(\t r)+0.*1.9570714950321337*sin(\t r)},{0.*1.9570714950321337*cos(\t r)+1.*1.9570714950321337*sin(\t r)});
\draw [shift={(3.453995848431412,0.)},line width=2.pt,color=ffffff]  plot[domain=-0.06785769590849178:0.011905417543477381,variable=\t]({1.*1.957071495032133*cos(\t r)+0.*1.957071495032133*sin(\t r)},{0.*1.957071495032133*cos(\t r)+1.*1.957071495032133*sin(\t r)});
\draw [shift={(3.453995848431412,0.)},line width=2.pt,color=ffffff]  plot[domain=0.3414782906133585:0.39595171397579226,variable=\t]({1.*1.957071495032133*cos(\t r)+0.*1.957071495032133*sin(\t r)},{0.*1.957071495032133*cos(\t r)+1.*1.957071495032133*sin(\t r)});
\begin{scriptsize}
\draw[color=ffffff] (0.7911491159238497,1.7975768723976244) node {$r_4$};
\draw[color=ffffff] (-0.17273060519615357,1.4639261997022492) node {$s_4$};
\end{scriptsize}
\end{tikzpicture} 
\caption{\footnotesize{$\Z$-cover of the surface of equation $Cos(2\pi u)+Cos(2\pi v)+Cos(2\pi w)=0$ within $\mathbb{T}^2\times \mathbb{R}$}}
\end{figure}
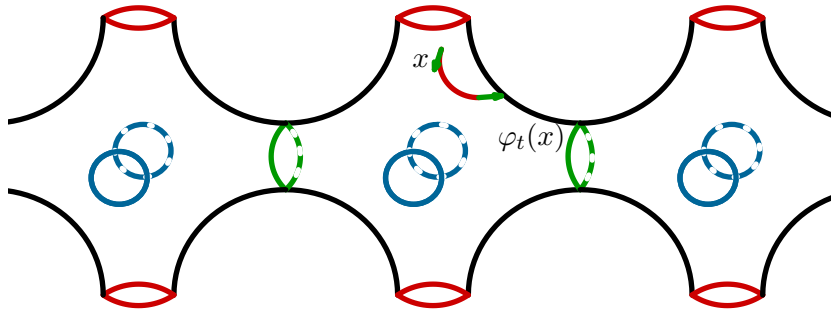
 
We start by stating our main results  in the particular cases 
of the $\mathbb Z$-periodic Lorentz gas and of the unit geodesic flow on a $\mathbb Z$-cover 
of a compact negatively curved surface.
We consider a natural projection $\Psi:\mathcal R_0\mapsto \mathbb R$ defined as follows.
For the Lorentz gas and for the unit geodesic flow on the particular example of $\mathbb Z$-periodic surface given above, we consider the function $\Psi$ which is the first coordinate (in $\mathbb R$) of the position of the particle. For more general geodesic flows on surfaces $\mathcal R_0$ not embedded in $\mathbb R\times \mathbb T^2$, this $\Psi$ can be replaced by
the label in $\mathbb Z$ of the copy in $\mathcal R_0$ of the compact surface $\mathcal R$ containing the position of the particle.

We will state results of convergence in the sense of {\bf strong convergence in distribution with respect to $\nu$. We will write
$\mathcal{L}_{\nu}$ for this type of convergence.} 
We recall that this means the convergence in distribution with respect to every probability measure absolutely continuous with respect to $\nu$. 
In the first perturbation setting (integrable case), we establish the following result (valid in a much more general context, see Theorem~\ref{thmgeneintegrable}).
\begin{thm}[First setting: integrable observable]\label{thmexempleintegrable}
Let $(\mathcal M,\f_t,\nu)$ be the above mentioned Lorentz gas flow or geodesic flow. 
Assume $\tilde f:\mathbb R^d\times\mathcal M\mapsto\mathbb R^d$ is
integrable in its second variable. 
Assume furthermore that $\tilde f$ is uniformly Lipschitz continuous in its first variable
with $\nu$-integrable Lipschitz constant. 
Then, for all $T>0$, 
\[
\left(\epsilon^{-\frac 12}(X_t^\epsilon(x_0,\omega)-x_0)\right)_{t\in[0;T]}\xrightarrow[ \epsilon \rightarrow 0]{\mathcal L_\nu,\|\|_\infty}
\left(\int_{\mathcal M}
\tilde f(x_0,\cdot)\, d\nu\,  \tilde L_t(0)\right)_{t\in[0;T]}\, ,
\]
where $(\tilde L_s(0))_{s\geq 0}$, is a continuous version, taken at position 0, of the local time of a Brownian motion $(\tilde B_s)_{s\geq 0}$ limit in distribution of the normalized projection of the trajectory 
\begin{align}\label{eqbrowniencontin}
    (\epsilon^{1/2}\Psi\circ \f_{t/\epsilon} )_{t\in[0,T]}\xrightarrow[ \epsilon \rightarrow 0]{\mathcal L_\nu,\|\|_\infty}(\tilde B_t)_{t\in [0;T]}\, .
\end{align}
Furthermore, if $\int_{\mathcal M}\tilde f(x_0,\cdot)\, d\nu=0$ and if $\tilde f(x_0,\cdot)$ is H\"older continuous, then $\left(\epsilon^{-\frac 34}(X_t^\epsilon(x_0,\omega)-x_0)\right)_{t\in[0;T]}$
has the same limit in distribution as $\left(\epsilon^{1/4}\int_0^{t/\epsilon}f(x_0,\varphi_s(\cdot))\, ds\right)_{t\in[0,T]}$  (see~\eqref{eqtilda}  and Theorem~\ref{thmExemple} with $\overline f=0$).
\end{thm}
In the second setting (when $\tilde f(x,\omega)=f(x,\omega)+\bar f(x)$) of a differential equation perturbed by a "centered" additive noise, we obtain the following result.
Whereas the result in the case of a non-centered additive perturbation term is fairly direct and general (see Theorem~\ref{thm:averagenonnull}), the result in the case of a non-centered additive perturbation term requires more specific assumptions (see Theorem~\ref{thm:mainflow} 
and Corollary~\ref{corbillard}).
\begin{thm}[Second setting: additive perturbation term]\label{thmExemple}
Let $(\mathcal M,\f_t,\nu)$ be the above mentioned Lorentz gas flow or geodesic flow. Suppose that  $\tilde f :\mathbb{R}^d \times \M \rightarrow \mathbb{R}^d$ is the sum of a $C^2_b$ offset map $\bar f : \mathbb{R}^d \rightarrow \mathbb{R}^d$ and of a map $f :\mathbb{R}^d \times \M \rightarrow \mathbb{R}$,  $C^2_b$ in its first coordinate.
\begin{itemize}
\item (Non-centered additive perturbation)
Assume furthermore that the Lipschitz constant $[f]$ and $[D_1f]$
of respectively $f$ and $D_1f$ with respect to their first coordinate are $\nu$-integrable. 
Then, for any $x\in\mathbb R^d$, the normalized error 
\[
\left(
\epsilon^{-\frac 12}E_t^\epsilon(x,.):= \epsilon^{-\frac 12}(X_t^\epsilon(x,.)-W_t(x))\right)_{t\geq 0}
\] 
converges strongly in distribution  with respect to $\nu$, for the \textbf{local topology} of $\mathcal C([0,+\infty))$, to the random process $(\tilde Y_t(x))_{t >0}$ solution of
\[
d\tilde Y_t(x,\cdot)=h(W_t(x))\, d\tilde L_t(0)+D\overline f(W_t(x))\tilde Y_tdt,\quad \tilde Y_0(x)=0\, ,
\]
with $h(x):=\int_{\mathcal M}f(x,\cdot)\, d\nu$.
\item (Centered additive perturbation, see Theorem~\ref{thmExempleprecis} for the precise assumptions)
Assume now that $f(x,\cdot)$ is
H\"older and with null integral with respect to $\nu$, 
with $f(x,\omega)$ rapidly decaying when $|\Psi(\omega)|$ tends to infinity.
Then, for any $x\in\mathbb R$ and $S>0$, the normalized error 
\[
\left(\epsilon^{-3/4}E_t^\epsilon(x,.):= \epsilon^{-3/4}(X_t^\epsilon(x,.)-W_t(x))\right)_{t\in[0,S]}
\] 
converges strongly in distribution with respect to $\nu$ (and for the uniform metric) to the random process $(Y_t(x,\cdot))_{t \in [0,S]}$ solution of the following stochastic differential equation.
\begin{align}
dY_t(x,\cdot)=\sqrt{\tilde a(W_s(x))}dB_{\tilde L_s(0)}+D\overline f(W_s(x))Y_sds,\quad Y_0(x)=0\, ,
\end{align}
where $B$ is a standard Brownian motion independent of $(\tilde L_s(0))_{s\geq 0}$
defined as in Theorem~\ref{thmexempleintegrable}, 
and where the map $\tilde a :\mathbb R \mapsto Mat_{d\times d}(\mathbb R)$ is the asymptotic variance matrix induced by $f$ : for any\footnote{
$\sqrt{\tilde a(W_s(x))}$ is  is the unique symmetric nonnegative matrix whose square power is 
$\tilde a(W_s(x))$,
i.e. the square root $\sqrt{\tilde a(W_s(x))}$ identifies, via reduction through orthogonal matrix, with the diagonal matrix  square root of the diagonal reduced matrix of $\tilde a(W_s(x))$.} $x\in \mathbb R$
\begin{align}\label{eqtilda}
    \frac{1}{T^{1/4}}\int_0^Tf(x,\f_s(.))ds\xrightarrow[ \epsilon \rightarrow 0]{\mathcal L_\nu,\|\|_\infty}\sqrt{\tilde a(x)} B_{\tilde L_1(0)}\, .
\end{align}
\end{itemize}
\end{thm}
Let us indicate that the two limit processes $\tilde Y_t(x)$ and $Y_t(x)$ appearing in Theorem~\ref{thmExemple} can be rewritten
\[
\tilde Y_t(x)=\tilde V_{t}(x)
+\int_0^t D\overline f(W_s(x)) \exp\left(\int_{s}^{t}D\overline f(W_u(x))\, du\right)\tilde V_{s}(x) ds\, ,
\]
and
\[
 Y_t(x)=V_{t}(x)
+\int_0^t D\overline f(W_s(x)) \exp\left(\int_{s}^{t}D\overline f(W_u(x))\, du\right)V_{s}(x) ds\, ,
\]
with 
\[
\tilde V_t(x)=\int_0^th(w_s(x))\, d\tilde L_s(0)\quad\mbox{and}\quad
V_t(x)=\int_0^t\sqrt{\tilde a(W_s(x))}\, dB_{\widetilde L_s(0)}\, .
\]

Observe that the first setting  corresponds to the second setting with $\bar f$ identically null so that  $W_s(x_0)=x_0$.

It is worth noticing that the flows we are considering here can be described as suspension flows of discrete time dynamics considered in \cite{MP24discrete}. 
Thus, Theorem~\ref{thmExemple}
is closely related to the study of the counterpart problem in collision dynamics studied in \cite{MP24discrete}. 
This extension of the averaging result \cite[Theorem 3.6]{MP24discrete} from the Poincar\'e section (collision dynamics
for the Lorentz gas) to the suspension flow was natural for probability preserving dynamical systems such as the Sinai billiard. Here, when $f$ is centered, the techniques used by Kifer or Pène (see \cite{kifer} or \cite{peneESAIMPS}) no longer holds mostly because of the order
of the error term $E_t^\epsilon$ is in $\epsilon^{3/4}$ whereas the one of the averaging of the hitting time $\sum_{k=0}^{\lfloor t/\epsilon \rfloor} (\tau \circ T^{\lfloor t/\epsilon\rfloor}- \overline \tau)$ is in $\epsilon^{1/2}$. This is why the statement of Theorem \ref{thmExemple} requires stronger assumptions on the first derivative of $f(\cdot)$ than its collision dynamics counterpart \cite[Theorem 3.6]{MP24discrete} for its proof to fit.\\
Theorem \ref{thmExemple} leads to the following result (see Corollary \ref{corbirkflo}) interesting in itself that provides a limit Theorem for a perturbed Birkhoff integral that generalises the Central Limit Theorem provided by Pène and Thomine in \cite{PT20} when $f$ depended only on the label of the cell in $\mathbb Z$. 


 \begin{thm}\label{thmexbirkflo}
Let $f: \mathbb R^d \times \M \mapsto \mathbb R^d$ 
satisfying the same assumptions as 
in the previous Theorem \ref{thmExemple}. Then the renormalized perturbed ergodic integral $u_t^\epsilon$ defined for any $\omega\in  \M$ by
$$
u_t^\epsilon(x,\omega):=\epsilon^{1/4}\int_0^{t/\epsilon} f(\epsilon s,\f_s(\omega))\, ds
$$
converges strongly in distribution with respect to $\nu$ to $(V_t=\int_0^{t} \sqrt{\tilde a(W_s(x))}dB_{\tilde L_s(0)})_{t\ge 0}$,
as $\epsilon\rightarrow 0+$.
 \end{thm}
Notice that such results is a continuous counterpart of \cite[Theorem 3.4]{MP24discrete} but for the same reason about the orders of the roof function $\tau-\bar \tau$ and of $v_t^\epsilon$ this result could not be reached as a mere consequence of the former (which was a statement paving the way to \cite[Theorem 3.4]{MP24discrete}) but as the last proof of the paper and a consequence of Theorem \ref{thmExemple}.

%
%
%
%
%
%
%

This paper is organized as follows. Section~\ref{secmodels} is dedicated to a more detailed presentation of the examples of Lorentz gas and geodesic flows. In Section~\ref{secintegrable}, we give a general result in the first setting where $\tilde f$ is integrable and apply it to prove  Theorem~\ref{thmexempleintegrable}. In Section~\ref{secaveragenonnull}, we state a general theorem the second setting when $\tilde f$ is non centered. 
In Section~\ref{sec:suspensionflow}, we present a general result in the second setting for suspension flows (this result is proved in Section~\ref{secproof}).
In Section~\ref{averagingproof}, we use the main result of
 Section~\ref{sec:suspensionflow} to study the case of non centered additive noise and prove in particular the first part of Theorem~\ref{thmExemple}. 
In Section~\ref{secZext0}, we apply the main result of
 Section~\ref{sec:suspensionflow} in a general context of $\mathbb Z$-extension, and 
and prove, in particular, the second part of Theorem~\ref{thmExemple}.

\section{Examples}\label{secmodels}
In this section, we introduce two important examples: 
the $\mathbb Z$-periodic Lorentz gas with finite horizon and the geodesic flow on a $\mathbb Z$-cover of a compact negatively curved surface. 

\subsection{$\Z$-periodic Lorentz gas with finite horizon.}\label{subsecbilliard}
The $\Z$-periodic Lorentz gas is 
the billiard system $(\M, \f_t, \nu)$ 
modeling the behavior of a point particle moving at unit velocity in a specific domain $\R_0$. This domain $\R_0$
corresponds to the flat cylindrical surface $\mathbb R\times\mathbb T$ doted of open convex obstacles belonging to a finite family $\{O_m+(l,0), l\in \Z, m\in I\}$  with $C^3$ boundary, and positive curvature periodically placed along the cylinder.  The point particle goes straight inside $\mathcal R_0$ and bounces against the obstacles according to the Snell-Descartes reflection law (the reflected angle is equal to the incident angle).\\

\begin{figure}[!htbp] 
\begin{center}
\definecolor{ffqqqq}{rgb}{1.,0.,0.}
\definecolor{qqqqff}{rgb}{0.,0.,1.}
\definecolor{ttzzqq}{rgb}{0.2,0.6,0.}
\begin{tikzpicture}[line cap=round,line join=round,>=triangle 45,x=0.38037453549730627cm,y=0.34978073522874864cm]
\clip(-16.882165919766276,-6.869686681735669) rectangle (16.365941027050763,5.683411002043446);
\draw [line width=2.pt] (0.,0.) ellipse (0.7607490709946125cm and 0.6995614704574973cm);
\draw [shift={(-5.,5.)},line width=2.pt]  plot[domain=-1.5707963267948966:0.,variable=\t]({1.*3.8*cos(\t r)+0.*3.8*sin(\t r)},{0.*3.8*cos(\t r)+1.*3.8*sin(\t r)});
\draw [shift={(-5.,-5.)},line width=2.pt]  plot[domain=0.:1.5707963267948966,variable=\t]({1.*3.8*cos(\t r)+0.*3.8*sin(\t r)},{0.*3.8*cos(\t r)+1.*3.8*sin(\t r)});
\draw [shift={(5.,-5.)},line width=2.pt]  plot[domain=1.5707963267948966:3.141592653589793,variable=\t]({1.*3.8*cos(\t r)+0.*3.8*sin(\t r)},{0.*3.8*cos(\t r)+1.*3.8*sin(\t r)});
\draw [shift={(5.,5.)},line width=2.pt]  plot[domain=3.141592653589793:4.71238898038469,variable=\t]({1.*3.8*cos(\t r)+0.*3.8*sin(\t r)},{0.*3.8*cos(\t r)+1.*3.8*sin(\t r)});
\draw [shift={(-15.,5.)},line width=2.pt]  plot[domain=-1.5707963267948966:0.,variable=\t]({1.*3.8*cos(\t r)+0.*3.8*sin(\t r)},{0.*3.8*cos(\t r)+1.*3.8*sin(\t r)});
\draw [shift={(-15.,-5.)},line width=2.pt]  plot[domain=0.:1.5707963267948966,variable=\t]({1.*3.8*cos(\t r)+0.*3.8*sin(\t r)},{0.*3.8*cos(\t r)+1.*3.8*sin(\t r)});
\draw [shift={(-5.,-5.)},line width=2.pt]  plot[domain=1.5707963267948966:3.141592653589793,variable=\t]({1.*3.8*cos(\t r)+0.*3.8*sin(\t r)},{0.*3.8*cos(\t r)+1.*3.8*sin(\t r)});
\draw [shift={(-5.,5.)},line width=2.pt]  plot[domain=3.141592653589793:4.71238898038469,variable=\t]({1.*3.8*cos(\t r)+0.*3.8*sin(\t r)},{0.*3.8*cos(\t r)+1.*3.8*sin(\t r)});
\draw [shift={(5.,5.)},line width=2.pt]  plot[domain=-1.5707963267948966:0.,variable=\t]({1.*3.8*cos(\t r)+0.*3.8*sin(\t r)},{0.*3.8*cos(\t r)+1.*3.8*sin(\t r)});
\draw [shift={(5.,-5.)},line width=2.pt]  plot[domain=0.:1.5707963267948966,variable=\t]({1.*3.8*cos(\t r)+0.*3.8*sin(\t r)},{0.*3.8*cos(\t r)+1.*3.8*sin(\t r)});
\draw [shift={(15.,-5.)},line width=2.pt]  plot[domain=1.5707963267948966:3.141592653589793,variable=\t]({1.*3.8*cos(\t r)+0.*3.8*sin(\t r)},{0.*3.8*cos(\t r)+1.*3.8*sin(\t r)});
\draw [shift={(15.,5.)},line width=2.pt]  plot[domain=3.141592653589793:4.71238898038469,variable=\t]({1.*3.8*cos(\t r)+0.*3.8*sin(\t r)},{0.*3.8*cos(\t r)+1.*3.8*sin(\t r)});
\draw [line width=2.pt] (10.,0.) ellipse (0.7607490709946125cm and 0.6995614704574973cm);
\draw [line width=2.pt] (-10.,0.) ellipse (0.7607490709946125cm and 0.6995614704574973cm);
\draw [line width=2.pt,color=ttzzqq] (-9.695508756736936,-3.4948625503364976)-- (-10.,-2.);
\draw [line width=2.pt,color=ttzzqq] (-10.,-2.)-- (-10.69,-5.);
\draw [line width=2.pt,color=ttzzqq] (-10.616120024427149,5.043159348380091)-- (-11.367798652172832,1.459152784020298);
\draw [line width=2.pt,color=ttzzqq] (-11.367798652172832,1.459152784020298)-- (-14.,1.);
\draw [->,line width=2.pt,color=qqqqff] (-9.695508756736936,-3.4948625503364976) -- (-9.861346128022495,-2.6807042411372577);
\draw [line width=2.pt,color=ffqqqq] (-15.,5.)-- (15.,5.);
\draw [line width=2.pt,color=ffqqqq] (-15.,-5.)-- (15.,-5.);
\draw (-1.740373498950506,1.2029536338021483) node[anchor=north west] {$O_1+0$};
\draw (-11.731474235636322,1.2029536338021483) node[anchor=north west] {$O_1-1$};
\draw (8.374840290364947,1.2029536338021483) node[anchor=north west] {$O_1+1$};
\draw (3.472374711494267,5.174571317950489) node[anchor=north west] {$O_2+1$};
\draw (3.5964877641239044,-2.7066075240313747) node[anchor=north west] {$O_2+1$};
\draw (12.470571027142983,5.050458265320853) node[anchor=north west] {$O_2+2$};
\draw (12.470571027142983,-2.7066075240313747) node[anchor=north west] {$O_2+2$};
\draw (-6.704895604136004,5.112514791635672) node[anchor=north west] {$O_2+0$};
\draw (-6.580782551506367,-2.7066075240313747) node[anchor=north west] {$O_2+0$};
\draw (-15.70309191978472,4.926345212691218) node[anchor=north west] {$O_2-1$};
\draw (-16.01337455135881,-2.7066075240313747) node[anchor=north west] {$O_2-1$};
\draw [->,line width=2.pt,color=qqqqff] (-14.,1.) -- (-14.787078582342174,0.8627045295666769);
\draw (-15.20663970926617,1.327066686431784) node[anchor=north west] {$\varphi_t(x)$};
\draw (-9.621552340932485,-2.7066075240313747) node[anchor=north west] {$x$};
\end{tikzpicture}
  \caption{\footnotesize{Illustration of a $\Z$-periodic Lorentz gas with finite horizon here with two patterns, $|I|=2$.}} \label{fig:GG08}
    \end{center}
      \vskip -0.5cm
\end{figure}
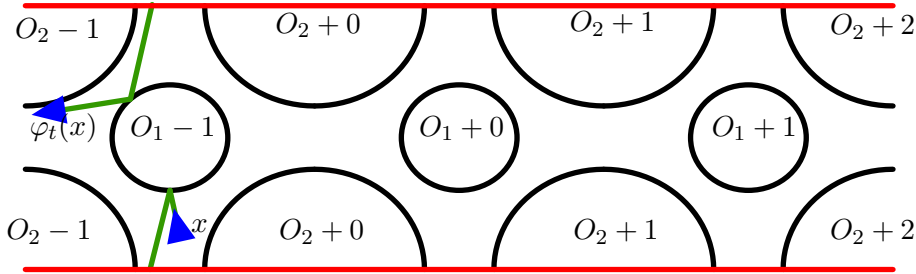

More precisely, given a finite family $(O_m)_{m\in I}$ of such obstacles on the cylinder $\mathbb R \times \mathbb T$ (where $\mathbb T:= \mathbb{R}/\Z$)
with $I$ a non finite set, and assuming that the closure of the obstacles $O_m+(l,0)$ remains pairwise disjoint,
we define the set $\R_0$ of allowed positions
$$
\R_0:=(\mathbb T\times \mathbb{R})\backslash \bigcup_{m,l}(O_m+(l,0)).
$$
$\M$ is then the phase space i.e the set of couple positions/unit-velocity 
 on $\R_0$:
$$
\M:=\R_0 \times \mathbb{S}^1/\sim
$$
where $\mathbb{S}^1$ is the unit circle and $\sim$ the relation identifying incident and reflected vectors, i.e. it identifies elements $(q,v) \in \partial \R_0 \times \mathbb{S}^1$ satisfying $\langle v, n(q) \rangle \geq 0$ (outgoing vector) with $(q,v')$ where $v'=v-2\langle v, n(q) \rangle n(q)$ and $n(q)$ denotes the unit normal vector to $\partial \R_0$ in $q$ directing into $\R_0$.

The Lorentz gas flow $(\f_t)_{t\geq 0}$
is then defined as the flow associating to the couple position/unit-velocity $(q_0,v_0) \in \M$ the new couple $(q_t,v_t)$ after time $t$. 
This flow preserves the infinite Lebesgue measure $\nu$ on $\M$.

To this system one can associate a natural Poincaré section $M$ corresponding to the phase space $\partial \R_0\times \mathbb S^1$ associated to the boundary of obstacles.
The Lorentz gas $(M,\f_t,\nu)$ is said to be in \textbf{finite horizon} when the map $\tau : M \to \mathbb R$ that corresponds to the hitting time of the Poincaré section is bounded. This can  equivalently be understood as the fact that there is no infinite free flight, any state in $\M$ leads to a collision with the Poincaré section $M$.  Notice that the finite horizon hypothesis and the disjointness of the closure of the obstacles lead the hitting time $\tau$ to be bounded away from $0$ and infinity.

The dynamical system $(\M,\f_t,\nu)$ can 
be expressed as a suspension flow over the finite horizon $\Z$-periodic billiard transformation $(M,T,\mu)$ with roof function $\tau$ whose transformation $T: M \to M$ maps a couple position/unit-velocity to the new couple position/unit-velocity after one reflection against an obstacle,
\begin{align*}
    T(x)=\f_{\tau(x)}(x).
\end{align*}
This transformation $T$ preserves an absolutely continuous infinite measure $\mu$ as recalled in \cite[Section 2.1]{MP22}. 
The $\Z$-periodic billard transformation with finite horizon can be seen itself as a $\Z$-extension with centered bounded\footnote{The boundedness of $\phi$ is a consequence of the finite horizon of the Lorentz gas.} step function $\phi : \Mp \mapsto \Z$ over the probabilistic dynamical system $(\Mp,\tp,\mup)$ called the Sinai billiard. The dynamical system $(\Mp,\tp,\mup)$ is the quotient modulo $\Z$ of the previous system $(M,T,\mu)$ and can be seen as a billiard transformation with obstacles on a torus. The ergodic properties of this Sinai billiard $(\Mp,\tp,\mup)$ involve ergodicity and mixing proved by Sinai in \cite{Sinai} as well as exponential mixing for smooth observable through the structure of Young towers proved by Young in \cite{Young}. Other stochastic properties have already been stated for this dynamical system such as Central limit theorem by Bunimovich, Sinai and Chernov (see \cite{sinaiclt}, \cite{buni_sinai_chernov}) and the already mentionned work of 
Pène on perturbed differential equations~\cite{peneESAIMPS} 
among others.
Both systems 
$(\M,\f_t,\nu)$ and 
$(M,T,\mu)$ 
share common stochastic properties that have been studied in previous decades. 
They were proved to be recurrent 
via~\cite{Schmidt} since the step function $\phi$ is $\overline\mu$-centered, their  
(conservative) ergodicity from \cite{penergodic} and \cite{Sim89ergobilliard}. 
The conservative ergodicity, combined with the Hopf ratio ergodic theorem, ensures that
\begin{equation}\label{Hopf1}
\forall G,H\in L^1(\mu),\quad \frac{\sum_{k=0}^{n-1}G\circ T^k}{\sum_{k=0}^{n-1}H\circ T^k}\xrightarrow[ n\rightarrow +\infty]{\mu-a.e.} \frac{\int_{M}G\, d\mu}{\int_{M}H\, d\mu}\, ,
\end{equation}
provided $\int_MH\, d\mu\ne 0$. 
It follows from (a direct adaptation) of~\cite{DSV08} (one can also apply~\cite[Theorem 2.6]{FP23}) that $(M,T,\mu)$ satisfies a Law of Large Numbers of the following form
\begin{equation}\label{LLLNbill1}
\forall T>0,\ \forall H\in L^1(\mu),\quad 
\left(\epsilon^{\frac 12}\sum_{k=0}^{\lfloor t/\epsilon\rfloor-1}
H\circ T^k\right)_{t\in[0,T]}\xrightarrow[ \epsilon \rightarrow 0]{\mathcal L_\mu,\|\|_\infty} \int_MH\, d\mu \, (L'_t(0))_{t\in[0,T]}\, ,
\end{equation}
where $(L'_t(0))_t$ is the (continuous version of the) local time, taken at 0, of the Brownian motion $B'$ limit of 
$(n^{-\frac 12}\sum_{k=0}^{\lfloor nt\rfloor -1}\phi\circ\overline T^k)_t$ as $n\rightarrow +\infty$. 
It is then straighforward to deduce, by ergodicity of $(\overline M,\overline T,\overline \mu)$, the analogous results 
for $(\mathcal M,\varphi,\nu)$, that is
\begin{equation}\label{Hopf2}
\forall g,h\in L^1(\nu),\quad \frac{\int_0^{t/\epsilon}g\circ \varphi_s\, ds}{\int_0^{t/\epsilon}h\circ \varphi_s\, ds}\xrightarrow[ \epsilon\rightarrow 0]{\nu-a.e.} \frac{\int_{\mathcal M}g\, d\nu}{\int_{\mathcal M}h\, d\nu}\, ,
\end{equation}
provided $\int_{\mathcal M}h\, d\nu\ne 0$, and 
\begin{equation}\label{FLLNbill2}
\forall T>0,\ \forall h\in L^1(\nu),\quad 
\left(\epsilon^{\frac 12}\int_0^{t/\epsilon}
h\circ \varphi_s\, ds\right)_{t\in[0,T]}\xrightarrow[ \epsilon \rightarrow 0]{\mathcal L_\nu,\|\|_\infty} \int_{\mathcal M}h\, d\nu \, (\tilde L_{t}(0))_{t\in[0,T]}\, ,
\end{equation}
with $\tilde L_t(0)=\overline\tau L'_{t/\overline\tau}(0)$. 
Furthermore a limit theorem of the following form
for some null integral observables 
has been stated in \cite{PT20} (for very specific observables and for $t=1$) completed by~\cite[Theorem 1.2 with $F(x,\omega)=f(\omega)$]{MP24discrete} 
\begin{equation}\label{TCLbill}
\forall T>0,\quad 
\left(\epsilon^{\frac 14}\sum_{k=0}^{\lfloor t/\epsilon\rfloor-1}
H\circ T^k\right)_{t\in[0,T]}\xrightarrow[ \epsilon \rightarrow 0]{\mathcal L_\mu,\|\|_\infty
} \sqrt{A(H)}\, (B_{L'_t(0)})_{t\in[0,T]}\, ,
\end{equation}
with $A(H):=\sum_{l\in\mathbb Z}\int_MH.H\circ T^{l}\, d\mu$, and $B$ a Brownian motion independent of $(L_t'(0))_{t\ge 0}$. 

\subsection{Geodesic flow on $\mathbb Z$-periodic negatively curved surfaces}
The Geodesic flow on a surface describes
the evolution of a point particle moving at unit velocity on the surface along the geodesic defined by its
initial position and velocity. Geodesic flows on a negatively curved surface have been studied for its dynamical properties since \cite{Anosov67}. 
We consider here the geodesic flow on a $\mathbb Z$-periodic surface $\mathcal R_0$ given by the $\mathbb Z$-cover of a compact negatively curved surface $\mathcal R$.  
We consider the measure preserving dynamical system  $(\M,\f,\nu)$ whose flow $\f$ is given by a geodesic flow over the space $\M:=T^1\R_0$ which is the unit tangent bundle over a $\Z$-cover $\R_0$ of a $C^3$ negatively curved compact surface $\R$ as stated in \cite[section 2.2]{MP22}. Furthermore, this geodesic flow preserves the Liouville measure $\nu$. Bowen and Ratner provided constructions of a Poincaré section $\Mp$ on $T^1\R$ that can be extended into a Poincaré section $M$ over the $\Z$-cover $\M$ and generating a discrete time measure invariant dynamical system $(M,T,\mu)$ with $T$ being the first return map to $M$ of the geodesic flow $\varphi$ 
and $\mu$ the Liouville measure adapted to the section $M$ and satisfying the following properties :
\begin{itemize}
    \item the geodesic flow $(\M,\f,\nu)$ is a suspension flow over $(M,T,\mu)$ with upper and lower bounded roof function $\tau$ corresponding to the first return time on the Poincaré section.
    \item $(M,T,\mu)$ is a $\Z$-extension with centered and bounded step function $\phi: \Mp \mapsto \Z$ over the $\Z$-quotient dynamical system $(\Mp,\tp,\mup)$ whose discrete-time dynamics $\tp$ corresponds to the first return map to $\overline M$ of the unit geodesic flow on $\mathcal R$ and preserves the Liouville measure adapted to $\Mp$.
    \item There is an isomorphism $\pi$ between $(\Mp, \tp, \mup)$ and a mixing two-sided subshift of finite type $(\Sigma, \sigma, m)$ (in particular $\Sigma\subset A^{\Z}$ with $A$ a finite set).
\end{itemize}
The dynamical system $(\Mp,\tp,\mup)$ can be endowed with the dynamical metric $d=d_\beta$ given by 
\begin{align}\label{eqdynamicmetric}
\forall x, y \in \Mp,\quad
    d(x,y):=\beta^{s_0(x,y)},
\end{align}
 where $s_0: \Mp \times \Mp \rightarrow \N$ is the separation time such that for $x,y\in \Mp$, $s_0(x,y)\geq n$ iff 
the $i^{th}$ coordinates of $\pi(x)$ and $\pi(y)$ seen as elements of $A^{\Z}$ coincide for any $i\in \{-n,\dots,n\}$.

For any $\eta\in(0,1)$, there exists $\beta\in(0,1)$ such that
the class of Lipschitz functions with respect to $d_\beta$  comprises the class of $\eta$-H\"older functions for the Riemann metric on hyperbolic surface restrained on the Poincaré section $\Mp$ (see \cite{ratner}).
We also introduce, for $k,n \in \Z$, the partitions $\xi_k^n$ of $\Mp$ such that two elements $x,y\in \overline M$ belong to the same atom of $\xi_{k}^n$ if $\pi(x)_i=\pi(y)_i$ for any $i\in \{k,\dots,n\}$. In particular, if $k=-n$ and $A\in \xi_k^n$,
$$
x,y\in A \Rightarrow s_0(x,y)\geq n.
$$
This coding structure being a mixing subshift of finite type, the dynamical system $(\Mp,\tp,\mup)$ and the step function $\phi$ verify \cite[Hypotheses 3.1 and 3.2]{MP24discrete}.
Furthermore the different limit theorems~\eqref{Hopf1},~\eqref{LLLNbill1},~\eqref{Hopf2},~\eqref{FLLNbill2} and~\eqref{TCLbill} stated for the Lorentz gas are still valid here (see \cite{thomine1,thomine2}).

\section{First setting: Case of integrable functions}\label{secintegrable}
\begin{thm}\label{thmgeneintegrable}
Let $(\varphi_t)_t$ be a flow defined on a measurable space $\mathcal M$ preserving an infinite ($\sigma$-finite) measure $\nu$. We assume furthermore that the dynamical system $(\mathcal M,\f_t,\nu)$ is conservative and ergodic. Let $x_0\in\mathbb R^d$. 
Assume $\tilde f:\mathbb R^d\times\mathcal M\mapsto\mathbb R^d$ is uniformly bounded, Lipschitz continuous in its first variable and integrable in its second variable. 
Assume furthermore that the Lipschitz constant of  $\tilde f$ in its first variable
is $\nu$-integrable. 
Then
\[
\sup_{t\in[0;T]}\left|X_t^\epsilon(x_0,\omega)-x_0-\int_0^t\tilde f(x_0,\varphi_{s/\epsilon}(\omega))\, ds\right|=o \left(\sup_{[0,T]}
\left|\int_0^t\tilde f (x_0,\varphi_{s/\epsilon}(\omega))\, ds\right|\right)
\, ,
\]
for $\nu$-almost every $\omega\in \mathcal M$.\\

If moreover $(\mathcal M,\f_t,\nu)$ satisfies a law of large numbers of the following form
\begin{align}\label{LLN}
\forall g\in L^1(\nu),\quad \mathfrak b_t \int_0^{t}g\circ\varphi_s\, ds
\xrightarrow[ t \rightarrow +\infty]{\mathcal L_\nu}\int_{\mathcal M}g\, d\nu\, Y\, ,
\end{align}
for some non degenerate real valued random variable $Y$,
then 
\begin{align*}
(\mathfrak b_{t/\epsilon}/\epsilon)(X_t^\epsilon(x_0,\cdot)-x_0)
\xrightarrow[ t \rightarrow +\infty]{\mathcal L_\nu}\int_{\mathcal M}
\tilde f(x_0,\cdot)\, d\nu \, Y\, .
\end{align*}
If $x_0$ is such that there exists a family of positive numbers $(\mathfrak a_\epsilon)_{\epsilon>0}$ such that
\begin{align}\label{FLLN}
\forall S>0,\quad 
\left(\mathfrak a_{\epsilon} \int_0^{t/\epsilon}\tilde f(x_0,\varphi_s(\cdot))\, ds\right)_{t\in[0;S]}
\xrightarrow[ \epsilon \rightarrow 0]{\mathcal L_\nu,\, \Vert\cdot\Vert_\infty}
\, (Z_t)_{t\in[0;S]}\, ,
\end{align}
where $Z$ is a continuous process, 
then 
\begin{align}\label{LT0}
\forall S>0,\quad 
\left((\mathfrak a_{\epsilon}/\epsilon)(X_t^\epsilon(x_0,\cdot)-x_0)\right)_{t\in[0;S]}
\xrightarrow[ \epsilon \rightarrow 0]{\mathcal L_\nu,\, \Vert\cdot\Vert_\infty} \left(
Z_t\right)_{t\in[0;S]}\, .
\end{align}
\end{thm}
\begin{rem}
It follows from the Hopf ratio ergodic theorem that if the convergence in~\eqref{LLN} holds true for some integrable observable with non null integral, then it holds also true for all integrable observable. If $\lim_{\epsilon\rightarrow 0}\mathfrak a_\epsilon=0$, then the same applies to the functional version of~\eqref{LLN} (given in~\eqref{FLLN} in the particular case where $g=\tilde f(x_0,\cdot)$). Note that the last part of Theorem~\ref{thmgeneintegrable} (with~\eqref{FLLN} and~\eqref{LT0}) can also be applied to null integral functions.
\end{rem}
\begin{proof}[Proof of Theorem~\ref{thmgeneintegrable}]
Using the definition of $X_t^\epsilon$, we start by noticing that
\begin{align*}
 X_t^\epsilon(x_0,\omega)-x_0 
&= \int_0^t\tilde f\left(X_s^\epsilon(x_0,\omega),\varphi_{s/\epsilon}(\omega)\right)\, ds\, ,
\end{align*}
and so
\begin{align}\label{AAA0}
\left\vert X_t^\epsilon(x_0,\omega)-x_0 
-\int_0^t \tilde f(x_0,\varphi_{s/\epsilon}(\omega))\, ds \right\vert\le\int_0^t[\tilde f(\cdot,\varphi_{s/\epsilon}(\omega)]\, |X_s^\epsilon(x_0,\omega)-x_0|\, ds\, ,
\end{align}
where we write here $[H]$ for the Lipschitz constant
of $H:\mathbb R^d\rightarrow \mathbb R^d$. 
Applying the Gr\"onwall lemma, we obtain the following
\begin{equation}\label{AAA1}
\sup_{t\in[0;S]}
\left\vert X_t^\epsilon(x_0,\omega)-x_0 \right\vert
\le \sup_{t\in[0;S]}\left|\int_0^t \tilde f(x_0,\varphi_{s/\epsilon}(\omega))\, ds\right| \exp\left(\int_0^S[\tilde f(\cdot,\varphi_{s/\epsilon}(\omega)]\,  ds\right)\, .
\end{equation}
Since $(\mathcal M,\f_t,\nu)$ is conservative and ergodic and $\nu(\mathcal M)=\infty$, it follows from the Hopf ratio ergodic theorem that
\begin{equation}\label{AAA2}
\int_0^S[\tilde f(\cdot,\varphi_{s/\epsilon}(\omega)]\,  ds=\epsilon\int_0^{S/\epsilon}[\tilde f(\cdot,\varphi_{u}(\omega)]\,  du
=T\frac{\int_0^{T/\epsilon}[\tilde f(\cdot,\varphi_{u}(\omega)]\,  du}{\int_0^{T/\epsilon}\mathbf 1_{\mathcal M}(u)\,  du}\underset{\epsilon \to 0}=o(1)\, ,
\end{equation}
for $\nu$-almost every $\omega\in\mathcal M$.
Thus, combining~\eqref{AAA0},~\eqref{AAA1} and~\eqref{AAA2}, we obtain that
\begin{equation}\label{AAA1b}
\sup_{t\in[0;T]}
\left\vert X_t^\epsilon(x_0,\omega)-x_0 -\int_0^t \tilde f(x_0,\varphi_{s/\epsilon}(\omega))\, ds\right\vert
\le o\left(\sup_{t\in[0;S]}\left|\int_0^t \tilde f(x_0,\varphi_{s/\epsilon}(\omega))\, ds\right| \right)\, ,
\end{equation}
for $\nu$-a.e.\ 
$\omega\in\mathcal M$. 
This ends the proof of the first point of the Theorem. For the two other points, we observe that
\[
\int_0^t\tilde f(x_0,\varphi_{s/\epsilon}(\omega))\, ds=\epsilon\int_0^{t/\epsilon}\tilde f(x_0,\varphi_{u}(\omega))\, du\ ,
\]

\end{proof}
\begin{proof}[Proof of Theorem~\ref{thmexempleintegrable}]
It follows e.g. from~\cite{DSV08} (resp. Theorem~\ref{thmExemple}) 
that~\eqref{FLLN} holds true with $\mathfrak a_{\epsilon}=\sqrt{\epsilon}$ (resp. $\mathfrak a_\epsilon=\epsilon^{\frac 14}$), 
and with $Z_t=\int_{\mathcal M}\tilde f(x_0,\cdot)\, d\nu\, \tilde L_t(0)$ (resp. $Z_t=\sqrt{\tilde a(x_0)}B_{\tilde L_t(0)}$).
\end{proof}
\begin{rem}
In the next Sections, we investigate the second setting (with perturbation or averaging). In that context, analogously to~\eqref{LT0}, we will also get the same normalization as the normalization of $\int_0^t\widetilde f(x_0,\varphi_{s/\epsilon}(\cdot))\, ds=\epsilon\int_0^{t/\epsilon}\tilde f(x_0,\varphi_s)\, ds$.
\end{rem}

\section{Averaging with a non centered perturbation term}\label{secaveragenonnull}
Let $(\mathcal M,\nu,(\varphi_t)_t)$ be a continuous time
infinite ($\sigma$-finite) measure preserving dynamical system. 
Let $f:\mathbb R^d\times\M\rightarrow\mathbb R^d$ and $\overline f:\mathbb R^d\rightarrow\mathbb R^d$ be two maps, both uniformly Lipschitz in their first coordinate $x\in\mathbb R^d$ and uniformly bounded, $f$ being measurable in its second coordinate $\omega\in\M$.
We consider the solution $(W_t)_{t\geq 0}$ of the following differential equation
\begin{align}\label{EDflowW}
\frac{dW_t}{dt} (x)=\overline{f}(W_t(x)),\quad W_0(x)=x\, , \quad\forall t \in \mathbb R_+ .
\end{align}
We perturb this equation by the flow $\varphi_t$
and study the
asymptotic behaviour, as $\epsilon$ goes to $0$, of the solution $X^\epsilon(x,\omega)$ of the 
perturbed differential equation:
\begin{align}\label{EDflow}
\frac{dX_t^\epsilon}{dt} (x,\omega)=f(X_t^\epsilon(x,\omega),\f_{t/\epsilon}(\omega))+\overline{f}(X_t^\epsilon(x,\omega)),\quad X_0^\epsilon(x,\omega)=x\, ,
\end{align}
where $x\in \mathbb R^d$ and $\omega\in\M$. 
This leads us to introduce the following hypothesis
on $f$ and $\bar f$.
\begin{hyp}\label{HHH1b}
Let $f:\mathbb R^d\times\mathcal M\rightarrow \mathbb R^d$ and $\bar f:\mathbb R^d\rightarrow \mathbb R^d$
be two measurable functions, both uniformly bounded and Lipschitz in their first coordinate $x\in\mathbb R^d$.
We consider $(W_t)_t$, $(X_t^\epsilon)_t$,
given by respectively~\eqref{EDflowW},~\eqref{EDflow}.
\end{hyp}

\begin{thm}\label{thm:averagenonnull}
Let $(\varphi_s)_s$ be a flow preserving an infinite ($\sigma$-finite) measure $\nu$ on a measurable space $\mathcal M$. Assume 
$(\mathcal M,\f_t,\nu)$ is conservative ergodic and satisfies a functional law of large numbers of the following form
\begin{align}\label{FLLN}
\forall g\in L^1(\nu),\quad 
\left(\mathfrak a_{\epsilon} \int_0^{t/\epsilon} g\circ\varphi_{ s}\, ds\right)_{t\in[0;T]}
\xrightarrow[ \epsilon \rightarrow 0]{\mathcal L_\mu,\, \Vert\cdot\Vert_\infty}
\int_{\mathcal M}g\, d\nu
\, (Z_t)_{t\in[0;T]}\, ,
\end{align}
where $(Z_t)_{t\ge 0}$ is a continuous increasing process. 
Assume Hypothesis~\ref{HHH1b}, and that $f(\cdot,\omega)$ and $\overline f$ are twice differentiable, with bounded derivative (uniformly in $\omega$) up to order 2. Assume furthermore that the Lipschitz constant $[f]$ and $[D_1f]$
of respectively $f$ and $D_1f$ with respect to their first coordinate are $\nu$-integrable (in $\omega\in\mathcal M$). 
Then 
for any $x\in\mathbb R^d$, the normalized error 
\[
\left(
(\mathfrak a_\epsilon/\epsilon)E_t^\epsilon(x,.):= (\mathfrak a_\epsilon/\epsilon)(X_t^\epsilon(x,.)-W_t(x))\right)_{t\geq 0}
\] 
converges strongly in distribution  with respect to $\nu$, for the \textbf{local topology} of $\mathcal C([0,+\infty))$, to the random process $(Y_t(x))_{t >0}$ given by
\[
\tilde Y_t(x)=\tilde V_{t}(x)
+\int_0^t D\overline f(W_s(x)) \exp\left(\int_{s}^{t}D\overline f(W_u(x))\, du\right)\tilde V_{s}(x) ds\, ,
\]
with 
\[
\tilde V_t(x)=\int_0^th(w_s(x))\, dZ_s\, ,
\]
with $h(x):=\int_{\mathcal M}f(x,\cdot)\, d\nu$.
\end{thm}
Theorem~\ref{thm:averagenonnull} will be proved in Section~\ref{averagingproof}. 
\section{Averaging for suspension flows}\label{sec:suspensionflow}
We consider the continuous time measure preserving dynamics given by the  suspension flow $(\M,\f_s,\nu)$ over 
an infinite ($\sigma$-finite) dynamical system $(M,T,\mu)$ with a measurable roof function $\tau:\Omega\rightarrow (0,+\infty)$ 
bounded above and away from $0$, i.e.
\begin{equation}\label{boundtau}
0<\inf_M\tau\le\sup_M\tau<\infty\, .
\end{equation}
We recall that this system is defined as follows,
\begin{align*}
\M:=\{(y,t)\in M\times[0,+\infty), \ t\in [0,\tau(y))\} \textit{ and } \nu=\mu \otimes Leb
\end{align*}
with $Leb$ being the standard Lebesgue measure. The flow is then defined by 
\begin{align*}
\f_t(y,s)=(T^{n_{t+s}(y)}(y),t-t_{n_{t+s}(y)}(y)),
\end{align*}
where $n_t(y):= \sup\{n \in \N, t_n(y)
\leq t\}$
and $t_m(y):=S_m\tau(y)=\sum_{k=0}^{m-1}\tau\circ T^k(y)$. 
\begin{hyp}\label{HHH0b}
Let $(M,T,\mu)$ be a conservative ergodic infinite ($\sigma$-finite) measure
preserving dynamical system. Let $\tau :M\rightarrow \mathbb [0,+\infty)$ be a measurable function such that $0<\inf_{M}\tau\le\sup_{M}\tau<\infty$. 
Let $(\mathcal M,\f_t,\nu)$ be the suspension flow over $(M,T,\mu)$
with roof function $\tau$.
\end{hyp}
We will assume Hypothesis~\ref{HHH1b}.
We introduce the maps $F:\mathbb R^d\times M\rightarrow\mathbb R^d$ and $\overline F:\mathbb R^d\rightarrow\mathbb R^d$ defined by
\begin{equation}\label{defF}
F(x,\omega):=\int_0^{\tau(\omega)}f(x,(\omega,s))ds\,\quad\mbox{and}\quad
\overline F(x):=\overline\tau\overline f(x)\, ,
\end{equation}
with $\overline\tau>0$ (corresponding to the averaged value of $\tau$). 
We also denote by $(w_t)_{t\geq 0}$ the function defined by $w_{t}:=W_{\overline\tau t}$ for any $t \in \mathbb R_+$. This map corresponds to the solution of the following differential equation
\begin{align}\label{EDflowW2}
\frac{dw_t}{dt} (x)=\overline{F}(w_t(x)),\quad w_0(x)=x\, .
\end{align}
This leads us to introduce the following hypothesis
on $f$ and $\bar f$ (reinforcement of Hypothesis~\ref{HHH1b}).
\begin{hyp}\label{HHH1}
Let $f:\mathbb R^d\times\mathcal M\rightarrow \mathbb R^d$ and $\bar f:\mathbb R^d\rightarrow \mathbb R^d$
be two measurable functions, both uniformly bounded and Lipschitz in their first coordinate $x\in\mathbb R^d$.
We consider $(W_t)_t$, $(X_t^\epsilon)_t$, $F$ and $(w_t)_t$
given by respectively~\eqref{EDflowW},~\eqref{EDflow},~\eqref{defF} and~\eqref{EDflowW2}.
\end{hyp}

\begin{defn}
We will say that a family of continuous processes $((x_t^\epsilon)_{t\ge 0})_{\epsilon >0}$ \textbf{locally converges} when $\epsilon$ tends to $0$ to $(x^0_t)_{t\ge 0}$ in $
\mathcal C([0,+\infty))
$ if, for any 
$T>0$, $((x_t^\epsilon)_{t\in[0;T]
})_{\epsilon >0}$ converges to $(x^0)_{t
\in[0,T]}$ in 
$
(
\mathcal C([0,T])
,\|.\|_\infty)
$. We call \textbf{local topology} the topology induced by this notion of local convergence.    
\end{defn}
In particular to check the local convergence of a process, it is enough to check its convergence for any compact subinterval.
The next result states how a limit theorem for $(E_t^\varepsilon)_t$ can be deduced from the study of
perturbed Birkhoff integrals of $(M,T,\mu)$.
\begin{thm}\label{thm:mainflow}
Assume Hypotheses~\ref{HHH0b} and~\ref{HHH1}.
Let $\overline\mu$ be a probability measure on $ M$
absolutely continuous with respect to $\mu$.  
Let  $(\mathfrak a_\epsilon)_{\epsilon>0}$ be a family of positive real numbers such that 
$\lim_{\epsilon\rightarrow 0}\mathfrak a_\epsilon=0$ and $\lim_{\epsilon\rightarrow 0}\epsilon/\mathfrak a_\epsilon=0$. 
Assume furthermore that 
\begin{itemize}
\item[(a)] the functions $F(.,\omega)$ and $\overline f$ are twice differentiable, uniformly bounded, and with uniformly bounded derivatives of first and second orders;
\item[(b)] the family of continuous processes $(v_t^\epsilon(x,\cdot):=\mathfrak a_\epsilon\int_0^{t/\epsilon} F(w_{\epsilon s}(x),T^{\lfloor s \rfloor})ds)_t$ converges in distribution to
 $(v_t(x,\cdot))_{t\geq 0}$ (with respect to $\overline\mu$ in $\mathcal C([0;S'])$ for all $S'>0$) and the family   $(\dot v_t(x,\cdot))_{t\geq 0}$ is tight (with respect to $\overline\mu$ in $\mathcal C([0;S'])$ for all $S'>0$);
\item[(c)] For all $S'$, $((\epsilon n_{t/\epsilon})_{t\in [0,S']})_{\epsilon}$
converges in probability, with respect to $\overline\mu$, to
$(t/\bar\tau)_{t\in[0,S']}$ (for the uniform metric),
\item[(d)] the families of continuous processes $(z_t^\epsilon:=\sqrt{\epsilon \mathfrak a_\epsilon }\int_0^{t/\epsilon}(\tau\circ T^{\lfloor s\rfloor}-\overline\tau)\overline f(w_{\epsilon s}(x))\, ds)_t$ 
and $(\dot z_t^\epsilon:=\epsilon\int_0^{t/\epsilon}(\tau\circ T^{\lfloor s\rfloor}-\overline\tau)D\overline f(w_{\epsilon s}(x))\, ds)_t$ 
converges in distribution to 0, with respect to $\overline\mu$, (respectively with respect to the uniform topology on $\mathcal C([0;S'])$ for all $S'>0$, and in the sense of finite distributions).
\end{itemize}
Then, 
for any $x\in\mathbb R^d$, the normalized error 
\[
\left(
(\mathfrak a_\epsilon/\epsilon)E_t^\epsilon(x,.):= (\mathfrak a_\epsilon/\epsilon)(X_t^\epsilon(x,.)-W_t(x))\right)_{t\geq 0}
\] 
converges strongly in distribution  with respect to $\nu$, for the \textbf{local topology} of $\mathcal C([0,+\infty))$, to the random process $(Y_t=y_{t/\overline \tau}(x))_{t >0}$ with
\[
Y_t(x)=v_{t/\overline \tau}(x)
+\int_0^t D\overline f(W_s(x)) \exp\left(\int_{s}^{t}D\overline f(W_u(x))\, du\right)v_{s/\overline\tau}(x) ds\, .
\]
\end{thm}

The statement of Theorem \ref{thmexbirkflo} that provides a limit Theorem for a class of non-stationnary Birkhoff integral then derives from the following corollary whose proof is a direct consequence of Theorem \ref{thm:mainflow}.

\begin{cor}[Convergence of Birkhoff integral for the flow]\label{corbirkflo}
Under the Assumptions of Theorem~\ref{thm:mainflow} with 
$\overline f=1$,
the renormalized perturbed ergodic integral $u_t^\epsilon$ defined for any $\omega\in  \M$ by
$$
u_t^\epsilon(\omega):=(\mathfrak a_\epsilon/\epsilon)
\int_0^{t/\epsilon} f(\epsilon s,\f_s(\omega))\, ds
$$
satisfies the following limit
\begin{align}\label{eqconvbirk}
\left(u_t^{\epsilon}\right)_{t \in [0,T]} \underset{\epsilon \rightarrow 0}{\overset{\mathcal{L}_{\nu}, \|.\|_\infty}{\rightarrow}} \left(v_{t/\overline\tau}\right)_{t \in [0,T]}\, ,
\end{align}
with $v_t$ as in Theorem~\ref{thm:mainflow}.

\end{cor}

\begin{proof}

    We simply consider the couple $(A_t^\epsilon,t)_{t\geq 0}$ solution of the differential equation
    \begin{align}
              \frac{d(A_t^\epsilon,t)}{dt}(\omega)=(f(t,\f_{t/\epsilon}(\omega)),1)\; A_0^\epsilon(\omega)=0 \; \forall t\in [0,S].
    \end{align}
Theorem \ref{thm:mainflow} applies for $x=0\in \mathbb R^{d+1}$ with the two maps
$(x,y,\omega)\mapsto (f(y,\omega),0)$ and $(x,y)\mapsto (0_{\mathbb R^d},1)$. With these choices,
$X_t^\epsilon(x,\cdot):=(A_t^\epsilon,t)
$ and $W_t(x):=(0,t)
$ and thus,
$(\mathfrak a_\epsilon/\epsilon)E_t^\epsilon(0,\cdot)=(u_t^\epsilon(\cdot),0)$. The conclusion of Theorem \ref{thm:mainflow} then provides equation \eqref{eqconvbirk} we were looking for.
\end{proof}

\section{Averaging with non centered perturbation term~: proofs
}\label{averagingproof}
\begin{thm}\label{verifb}
Assume Hypotheses~\ref{HHH0b} and~\ref{HHH1}. Assume 
$(M,T,\mu)$ is conservative ergodic and satisfies a functional law of large number of the following form
\begin{align}\label{FLLN}
\forall G\in L^1(\mu),\quad 
\left(\mathcal S_\epsilon(G)(t):=\mathfrak a_{\epsilon} \int_0^{t/\epsilon} G(T^{\lfloor s\rfloor}(\cdot))\, ds\right)_{t\in[0;T]}
\xrightarrow[ \epsilon \rightarrow 0]{\mathcal L_\mu,\, \Vert\cdot\Vert_\infty}
\int_{M}G\, d\mu
\, (Z_t)_{t\in[0;T]}\, ,
\end{align}
Suppose that Assumption~(a) of Theorem~\ref{thm:mainflow} holds true and that the Lipschitz constant $[F]$ and $[D_1F]$
of respectively $F$ and $D_1F$ with respect to their first coordinate are $\mu$-integrable (in their last coordinate).
Then Assumption (b) of Theorem~\ref{thm:mainflow} holds true with $\mathfrak a_\epsilon$ and with $v_t(x,\cdot)=\int_0^t H(w_u(x))\, dZ_s$
and $\dot v_t(x,\cdot)=\int_0^t D_1 H(w_u(x))\, dZ_s$, with
$H(x):=\int_{M}F(x,\cdot)\, d\mu$, furthermore  $D_1H(x)=\int_{M}D_1F(x,\cdot)\, d\mu$.
\end{thm}
\begin{rem}
Note that, the $\mu$-integrability of $[F]$ and $[D_1F]$
follows from the $\nu$-integrability of $[f]$ and $[D_1f]$.
Indeed, $[F](\omega)\le \int_0^{\tau(\omega)}[f](\omega,s)\, ds$ and it follows from the Lebesgue dominated convergence theorem combined with the integrability of $[f]$  that $D_1F(x,\omega)=\int_0^{\tau(\omega)}D_1f(x,(\omega,s))\, ds$, and so $[D_1F](\omega)\le \int_0^{\tau(\omega)}[D_1f](\omega,s)\, ds$. Then, we also have $H(x)=\int_{\mathcal M}f(x,\cdot)\, d\nu$ and $D_1H(x)=\int_{\mathcal M}D_1f(x,\cdot)\, d\nu$. 
\end{rem}
\begin{proof}
Recall that $[F](\omega)$ (resp. $[D_1F](\omega)$) is the Lipschitz constant of $F(\cdot,\omega)$ (resp. $D_1F(\cdot,\omega)$) for all $\omega\in M$. 
First, the identity $D_1H(x):=\int_{\mathcal M}D_1F(x,\cdot)\, d\mu$ 
follows from the Lebesgue dominated convergence theorem 
combined with the $\mu$-integrability of $[F]$. Now let us prove Assumption~(b) of Theorem~\ref{thm:mainflow}. 
For any positive integer $M$, 
we observe that
\begin{equation}\label{control}
\left| v_t^\epsilon(x,\cdot)-\int_0^{t/\epsilon}F(w_{\lfloor\epsilon s M\rfloor/M}(x),T^{\lfloor s\rfloor})\, ds
\right|\le  M^{-1}\Vert \overline F\Vert_\infty\int_0^{t/\epsilon}[F]\circ T^{\lfloor s\rfloor}\, ds\, .
\end{equation}
We obtain an analogous estimate by replacing $(v,F)$ by $(\dot v,D_1F)$ in the above formula. 
Thus, considering $\overline \mu$ a probability measure absolutely continuous with respect to $\mu$ and applying the law of large numbers to both $[F]$ and $[D_1F]$, 
we conclude from~\eqref{control} and from the analogous estimate for $(\dot v,D_1F)$ that, for all $\eta>0$, 
\begin{equation}\label{control}
\lim_{M\rightarrow +\infty}\limsup_{\epsilon\rightarrow 0}\overline\mu\left(
\sup_{t\in[0,S]}\mathfrak a_\epsilon\left| v_t^\epsilon-\int_0^{t/\epsilon}F(w_{\lfloor\epsilon s M\rfloor/M},T^{\lfloor s\rfloor})\, ds
\right|>\eta\right)=0\, ,
\end{equation}
and
\begin{equation}\label{controlbis}
\lim_{M\rightarrow +\infty}\limsup_{\epsilon\rightarrow 0}\overline\mu\left(
\sup_{t\in[0,S]}\mathfrak a_\epsilon\left| \dot v_t^\epsilon-\int_0^{t/\epsilon}D_1F(w_{\lfloor\epsilon s M\rfloor/M},T^{\lfloor s\rfloor})\, ds
\right|>\eta\right)=0\, .
\end{equation}
Furthermore, for $G=F$ or $G=D_1F$, for all positive integer $M$ and all $S>0$ and all $t\in[0,S]$,
\[
\mathfrak a_\epsilon\int_0^{t/\epsilon}G(w_{\lfloor\epsilon s M\rfloor/M}(x),T^{\lfloor s\rfloor})\, ds
=\mathcal H_{S,N}\left(\mathcal S_\epsilon(G(w_{k/M}(x)),k=0,...,\lfloor SN\rfloor)\right)(t)\, ,
\]
with $\mathcal H_{S,N}:\mathcal C([0,S)\rightarrow (\mathbb R^d)^{\lfloor SN\rfloor+1})\rightarrow \mathcal C([0,S)\rightarrow \mathbb R^d)$
is given by
\[
\mathcal H_{S,N}(g_0,...,g_{\lfloor SN\rfloor})(t)
:=
g_{\lfloor tN\rfloor}(t)-g_{\lfloor tN\rfloor}(\lfloor tN\rfloor/N)+\sum_{k=0}^{\lfloor tN\rfloor -1}(g_k((k+1)/N)-g_k(k/N))\, .
\]
Observe that $\mathcal H_{S,N}$ is continuous, thus
$
\left(\mathfrak a_\epsilon\int_0^{t/\epsilon}G(w_{\lfloor\epsilon s N\rfloor/N}(x),T^{\lfloor s\rfloor})\, ds\right)_{t\in[0,T]}$
converges in distribution to $G_M$ given by
\begin{align*}
G_M(t)&:=\mathcal H_{S,N}\left( (\int_{M}G(w_{k/N}(x,\cdot))d\mu Z_t,k=0,...,\lfloor SN\rfloor) \, \right)(t)\\
&=\int_{M}G(w_{\lfloor tM\rfloor/N}(x),\cdot)\, d\nu (Z_t-Z_{\lfloor tN\rfloor/N})+\sum_{k=0}^{\lfloor tN\rfloor -1}
\int_{\M}G(w_{k/N}(x),\cdot)\, d\mu(Z_{(k+1)/N}-Z_{k/N})\\
&=\int_0^t
\int_{M}G(w_{\lfloor sN\rfloor/N}(x),\cdot)\, d\mu\,  dZ_s\, ,
\end{align*}
We conclude by noticing that this random process converges to
$(\int_0^t\int_{M}G(w_s(x),\omega)\, d\mu(\omega)\, dZ_s)_t$.
\end{proof}
\begin{proof}[Proof of Theorem~\ref{thm:averagenonnull}]
This result can be proved directly by following the steps of Theorem~\ref{thm:mainflow}. We can also see it as a direct application
of Theorem~\ref{thm:mainflow} applied with $(M,T,\mu)=(\mathcal M,\nu,\varphi_1)$ and $\tau=\overline\tau=1$, combined with Theorem~\ref{verifb}. Indeed, considering the projection $\pi:\mathcal M\times[0,1]\rightarrow \mathcal M$ given by $\Pi(\omega,s)=\varphi_s(\omega)$, we observe that $\nu$ is the image measure of $\nu\otimes Leb_{[0,1]}$ by $\Pi$ and that $ \varphi_s\circ \Pi=\Pi\circ \tilde\varphi_s$, where $\tilde\varphi_s$ is the special flow defined on $\mathcal M\times[0,1]$ with roof function $\tau=1$.
Since $\tau=\overline\tau=1$, Assumptions~(c) and~(d) of Theorem~\ref{thm:mainflow} are trivially satisfied with $n_{t/\epsilon}=\lfloor t/\epsilon\rfloor$ and $z^\epsilon=\dot z^\epsilon=0$. Assumption~(a) comes from the fact that $F(x,\omega)=\int_0^1f(x,\varphi_s(\omega))\, ds$ and from the assumptions on $f$ combined with the dominated convergence theorem. Finally Assumption~(b) comes from Theorem~\ref{verifb}. 
\end{proof}
\begin{proof}[Proof of the first part of Theorem~\ref{thmExemple}]
This follows directly from this result, since the dynamical systems considered in Section~\ref{secmodels}
are conservative ergodic and due to the functional Law of Large Numbers given by~\eqref{FLLNbill2} in the case of the Lorentz gas and the analogous statement for the geodesic flow on a $\mathbb Z$-cover of a negatively curved compact surface.
\end{proof}

\section{Averaging for $\mathbb Z$-periodic flows and centered perturbation term}
\label{secZext0}
We consider here the case of continuous time dynamical systems  $(\M,\f_s,\nu)$
satisfying some $\mathbb Z$-periodicity. It will be modeled
by a suspension flow over a $\mathbb Z$-extension $(M=\overline M\times \mathbb Z,T,\mu)$ with a measurable $\mathbb Z$-invariant roof function, i.e. satisfying $(\omega,m)\mapsto\tau(\omega,m)=\tau(\omega)$. 
Our general context will include the two examples of flows introduced in Section~\ref{secmodels}. 

\subsection{General model of $\Z$-extension.}\label{secZext}

In the rest of the paper we consider an infinite measure preserving dynamical system $(M,T,\mu)$ given by a $\Z$-extension of an ergodic probability preserving dynamical system $(\Mp,\tp,\mup)$ with a centered bounded step function $\phi :\Mp \rightarrow \Z$. The dynamics of $(M,T,\mu)$  is described by the following skew product for any $(\omega,m)\in M:=\Mp \times \Z$ by :
\begin{align*}
    T(\omega,m)=(\tp (\omega),m+\phi(\omega)).
\end{align*}
When iterated, the dynamics brings the following identity
\begin{align*}
    T^n(\omega,m)=(\tp^n (\omega),m+S_n\phi(\omega)),
\end{align*}
where $S_n\phi:=\sum_{k=0}^{n-1}\phi\circ\tp^k$. The map $T$ preserves the product measure $\mu:=\mup \otimes m$  on  $\Mp\times \Z$  where $m$ is the counting measure on $\Z$. 
The assumption that $\phi$ is centered (i.e $\int_{\Mp} \phi\,  d \mup=0$) ensures that the Birkhoff sum $S_n\phi$ is recurrent as a random walk on $\Z$ and since $(\Mp,\tp,\mup)$ is ergodic, it implies that $(M,T,\mu)$ is a recurrent dynamical systems (see \cite{Schmidt}).
We summarize below the above assumptions on the dynamical system $(\mathcal M,\f_t,\nu)$.
\begin{hyp}\label{HHH0}
Let $(\Mp,\mup,\tp)$ be an ergodic probability preserving dynamical system. Let $\phi:\Mp\rightarrow\mathbb Z^d$
and $\tau :\Mp\rightarrow \mathbb (0,+\infty)$ be two bounded measurable functions with $\phi$ centered and $0<\inf_{\Mp}\tau\le\sup_{\Mp}\tau<\infty$. Let $(M,T,\mu)$ be the $\mathbb Z$-extension of $(\Mp,\mup,\tp)$ by $\phi$ and $(\mathcal M,\f_t,\nu)$ be the suspension flow over $(M,T,\mu)$
with roof function $\tau:(\omega,m)\mapsto \tau(\omega,m)=\tau(\omega)$.
We further assume that $(M,T,\mu)$ is conservative ergodic.
In this setting we identify $\overline\mu$ with the measure $\overline\mu\otimes\delta_0$ which is absolutely continuous with respect to $\mu$. We set
\[
\overline\tau:=\int_{\Mp}\tau\, d\mup\, .
\]

\end{hyp}
Note that in this context,
\begin{equation}\label{Smtau}
t_m(\omega,\ell)=S_m\tau(\omega,\ell)=
\sum_{k=0}^{m-1}\tau\circ T^k(\omega,\ell)=\sum_{k=0}^{m-1}\tau\circ \overline T^k(\omega)\, .
\end{equation}
\begin{cor}\label{corbillard}
Assume Hypotheses~\ref{HHH0} and~\ref{HHH1} and Assumptions~(a), (b) and (d)
 of Theorem~\ref{thm:mainflow} with the normalization
$\mathfrak a_\epsilon=\epsilon^{\frac 14}$ and the limit process
\begin{equation}\label{vu}
v_{u
}(x,\cdot)= \int_0^{u
} \sqrt{a(w_s(x))}\, dB_{L_s'(0)}
\end{equation}
for some continuous function $a:\mathbb R^d\rightarrow S_d(\mathbb R)
$ (with $S_d(\mathbb R)$ the set of real valued nonnegative symmetric matrices), where $(L'_t(x))_{x\in\mathbb R,t\geq 0}$ is a continuous version of the local time associated to the Brownian motion $B'$ given by the following  limit
\begin{equation}\label{Gaussianlimit}
\forall S>0,\quad 
(S_{\lfloor nt\rfloor}\phi/\sqrt{n})_{t\in[0,S]} \xrightarrow[ \epsilon \rightarrow 0]{\overline\mu,\|\|_\infty} (B'_t)_{t\in[0,S]}\, ,
\end{equation}
and where $B$ is a standard Brownian motion independent of $B'$.
Then, setting $\tilde a:=a/\bar\tau$, 
for any $x\in\mathbb R^d$,
the normalized error 
\[
\left(\epsilon^{-3/4}E_t^\epsilon(x,.):= \epsilon^{-3/4}(X_t^\epsilon(x,.)-W_t(x))\right)_{t\geq 0}
\] 
converges strongly in distribution  with respect to $\nu$ and for the \textbf{local topology} of $\mathcal C([0,+\infty))$
to the random process $(Y_t(x,\cdot))_{t \ge 0}$ given by
\begin{align}
&Y_t(x,\cdot)=  \int_0^{t} \sqrt{\tilde a(W_s(x))}\, dB_{\tilde L_s(0)}
\\
& +\int_0^t D\overline f(W_s(x))  \int_0^{s} \sqrt{\tilde a(W_u(x))}\, dB_{\tilde L_u(0)}(\omega)\exp\left(\int_{s}^{t}D\overline f(W_u(x))du\right)\, ds,
\end{align}
where the process $(\tilde L_s(0))_{s\geq 0}$ 
is a continuous version of the local time (taken at 0) of a Brownian motion $(\tilde B_s=B'_{s/\bar\tau})_{s\geq 0}$ 
limit in distribution 
of the following Birkhoff sum :
\begin{align}\label{eqbrowniencontin}
\forall S>0,\quad 
(\epsilon^{1/2}S_{n_{t/\epsilon}}\phi)_{t\in[O,S]} \xrightarrow[ \epsilon \rightarrow 0]{\mathcal L_\nu,\|\|_\infty}(\tilde B_t)_{t\in[0,S]},
\end{align}
with $n_t:= \sup\{n \in \N\, :\, \sum_{k=0}^{n-1}\tau\circ T^k\leq t\}$.
Furthermore, for any $x\in \mathbb R^d$,
if $a(x)$ is the asymptotic variance matrix of $F$ in the following sense
\begin{equation}\label{TCLinf}
    \left(\epsilon^{1/4}\sum_{k=0}^{\lfloor t/\epsilon\rfloor-1}F(x,T^k(.))\right)_t\xrightarrow[ n \rightarrow +\infty]{\mathcal L_\nu,\|\|_\infty}\left(\sqrt{ a(x)}B_{
     L_t(0)}\right)_t\, ,
\end{equation}
then $\tilde a(x)=a(x)/\overline\tau$ is the asymptotic variance of $f$ given by
\begin{align}\label{eqtilda}
    \left(\epsilon^{1/4}\int_0^{t/\epsilon}f(x,\f_s(.))\, ds\right)_t\xrightarrow[ \epsilon \rightarrow 0]{\mathcal L_\nu,\|\|_\infty}\left(\sqrt{\tilde a(x)}B_{
    \tilde L_t(0)}\right)\, .
\end{align}
\end{cor}
Before proving Corollary~\ref{corbillard} (the proof is given in Subsection~\ref{sec:proofcoro}), we
make some comments on Assumptions~(a), (b) and (d) of Theorem~\ref{thm:mainflow} 
in the present context of $\mathbb Z$-periodic flow satisfying Hypothesis~\ref{HHH0} and we apply Corollary~\ref{corbillard} to prove Theorem~\ref{thmExemple}.
\begin{itemize}
\item Recall that Assumption~(a) of Theorem~\ref{thm:mainflow}
deals with the smoothness of $\overline f:\mathbb R^d\mapsto\mathbb R^d$ and of $F:\mathbb R^d\times M\rightarrow \mathbb R^d$ given by
$F(x,\omega):=\int_0^{\tau(\omega)}f(x,(\omega,s))\, ds$. Hence, this assumption is related to the smoothness of $f$ in its first variable $x\in\mathbb R^d$.
\item 
If $(M,T,\mu)$ and $F$
satisfy the assumptions of \cite[Theorem 3.4]{MP24discrete} in which some discrete time slow fast systems (perturbed by $(M,T,\mu)$) are studied, then
Assumption~(b) holds true with $\mathfrak a_\epsilon=\epsilon^{\frac 14}$ and with $v$ given by~\eqref{vu}
and with $a(x)$ the asymptotic variance matrix
of $F(x,\cdot)$ given by the following sum of the self-covariance matrices of the stationary process $(F(x,T^{l}(\cdot)))_l$
\[
a(x):=\left(\frac 12\sum_{l \in \Z}
\int_{M}[F_i(x,.)F_j(x,T^{|l|}(.))+F_j(x,.)F_i(x,T^{|l|}(.))]\, d\mu
\right)\, ,
\]
for any $x \in \mathbb R^d$.
\item Finally, due to~\eqref{Smtau}, Assumption~(d) of Theorem~\ref{thm:mainflow} deals with the error term in 
averaging of the respective perturbed differential equations
\[\frac{d(z_t^\epsilon,w_t)(x,\omega)}{dt}=\left(\tau(\tp^{\lfloor t/\epsilon\rfloor}(\omega))\overline f(w_t(x)),\overline F(w_t(x))\right)
\]
and
\[\frac{d(\dot z_t^\epsilon,w_t)(x,\omega)}{dt}=\left(\tau(\tp^{\lfloor t/\epsilon\rfloor}(\omega))\overline Df(w_t(x)),\overline F(w_t(x)) \right)\, .
\]
These are  differential equation perturbed by the probability preserving dynamical system $(\Mp,\mup,\tp)$).  
It was proved in~\cite{kasminski,kifer,peneESAIMPS} that, under general assumptions, $(z^\epsilon/\sqrt{\epsilon})_\epsilon$ and $(\dot z^\epsilon/\sqrt{\epsilon})_\epsilon$ converge in distribution. In particular
Assumption~(d) holds true under \cite[Hypothesis 3.1]{MP24discrete} provided $\tau$ is Lipschitz.
\end{itemize}
In particular the conclusions of Corollary \ref{corbillard} hold true on examples given in the previous section \ref{secmodels}, namely the $\Z$-periodic Lorentz gas and the geodesic flow over a $\Z$-cover of a smooth negatively curved compact surface, thus leading to the statement of Theorem \ref{thmExemple} in the introduction. 
More precisely, as a consequence of Corollary~\ref{corbillard} and of~\cite[Theorem 1.2]{MP24discrete}, we obtain

\begin{thm}[Precise statement of the second part of  Theorem~\ref{thmExemple}]\label{thmExempleprecis}
Let $(\M,\f_t,\nu)$ be one of the two examples considered in Section~\ref{secmodels} (the $\Z$-periodic Lorentz gas flow with finite horizon or the geodesic flow on a $\Z$-cover of a compact negatively curved $C^2$ surface) and suppose that $f: \mathbb R^d\times\M \to \mathbb R^d$ and $\overline f : \mathbb R^d \to \mathbb R^d$ are two bounded $C^2$ maps in their first coordinate with bounded successive derivatives such that $f$ is 
$\eta$-H\"older continuous
in its second variable and  there exists $\e>0$ such that
\begin{itemize}
    \item[(i)] for any $x\in \mathbb R^d$, $\int_{\M} f(x,.)d\nu=0$,
    \item [(ii)] $\sup_{x\in\mathbb R^d}\sum_{m\in\mathbb Z}|1+|m||^{2(1+\e)}[\sup_{\mathcal M_m}| f(x,\cdot)|+\sup_{\mathcal M_m}| D_1f(x,\cdot)]<\infty$,
    where we write $\mathcal M_m$ for the set of $(\omega,m',s)\in\mathcal M$ such that $m'=m$,
    \item[(iii)] $\sup_{x\in \mathbb R^d} \sum_{m\in\mathbb Z}
    [f(x,\cdot)]_{H(\mathcal M_m)}
<\infty$
where we write $[h]_{H(\mathcal M_m)}$ for the $\eta$-H\"older constant of $h_{|\mathcal M_m}$. 
\end{itemize}
Then, for any $x\in\mathbb R^d$ and $S>0$, the normalized error 
\[
\left(\epsilon^{-3/4}E_t^\epsilon(x,.):= \epsilon^{-3/4}(X_t^\epsilon(x,.)-W_t(x))\right)_{t\in[0,S]}
\] 
converges strongly in distribution with respect to $\nu$ (and for the uniform metric) to the random process $(Y_t)_{t \in [0,S]}$ given by
\begin{align}
&Y_t(x,\omega)=  \int_0^{t} \sqrt{\tilde a(W_s(x))}dB_{\tilde L_s(0)}(\omega)\\
& +\int_0^t D\overline f(W_s(x))  \int_0^{s} \sqrt{\tilde a(W_u(x))}dB_{\tilde L_u(0)}(\omega)\exp\left(\int_{s}^{t}D\overline f(W_u(x))du\right)ds,
\end{align}
where the process $(\tilde L_s(0))_{s\geq 0}$ is 
a continuous version of the local time, taken at 0, of the Brownian motion $(\tilde B_s)_{s\geq 0}$ appearing 
as the following limit
\begin{align}\label{eqbrowniencontin1}
    \left(\epsilon^{1/2}\Psi\circ \f_{t/\epsilon}\right)_{t\geq 0} \xrightarrow[ \epsilon \rightarrow 0]{\mathcal L_\nu,\|\|_\infty}(\tilde B_t)_{t\geq 0},
\end{align}
and the map $\tilde a=a/\overline\tau
$ is the asymptotic variance induced by $f$, i.e. for any $x\in \mathbb R^d$,
\begin{align}\label{eqtilda1}
    \frac{1}{T^{1/4}}\int_0^Tf(x,\f_s(.))ds\xrightarrow[ \epsilon \rightarrow 0]{\mathcal L_\nu,\|\|_\infty}\sqrt{\tilde a(x)}\tilde B_{\tilde L_1(0)}.
\end{align}
\end{thm}
\begin{proof}
We will apply Corollary~\ref{corbillard}. 
Assumption~\eqref{Gaussianlimit} has been proved in~\cite{sinaiclt,buni_sinai_chernov} for the Lorentz gas flow and in~\cite{ratner} for the geodesic flow. 
Since $(S_{n_t}\phi-\Psi\circ\varphi_t)_t$ is uniformly bounded,
~\eqref{eqbrowniencontin1} will follow from~\eqref{eqbrowniencontin}. 
Assumption~(a) of Theorem~\ref{thm:mainflow} follows from our smoothness assumptions on $f$ and $\overline f$. As explained above, Assumption~(c) of Theorem~\ref{thm:mainflow} follows from in~\cite{kasminski,kifer,peneESAIMPS}.
Finally, let us prove that Assumption~(b) of Theorem~\ref{thm:mainflow} follows from~\cite[Theorem 3.4]{MP24discrete}. The dynamical system $(\Mp,\mup,\tp)$
satisfying~\cite[Hypothesis 3.2]{MP24discrete} and $\phi$ being bounded, it remains to check the four items of~\cite[Theorem 3.4]{MP24discrete} hold true for the observables $F$ and $D_1F$
respectively. Recall that $F(x,\omega)=\int_0^{\tau(\omega)}f(x,\omega,s)\, ds$. 
First, for every $x\in\mathbb R^d$, the fact that $f(x,\cdot)$ has null $\nu$-integral implies that $F(x,\cdot)$ has null $\mu$-integral.  The fact that $F$ and $D_1F$ are uniformly Lipschitz in their first variable follows from
the fact that $f$ and $D_1f$ are so. Assumption (ii) of Theorem~\ref{thmExempleprecis}  
ensures that
\[
\sup_{x\in\mathbb R^d}\sum_{m\in\mathbb Z}|1+|m||^{2(1+\e)}\left[\Vert F(x,\cdot,m)\Vert_\infty+\Vert D_1F(x,\cdot,m)\Vert_\infty\right]<\infty\, ,
\]
which gives the third item of ~\cite[Theorem 3.4]{MP24discrete} for both $F$ and $D_1F$.\\
Finally, for all $x\in\mathbb R^d$ and $m\in\mathbb Z$, and all
$\omega,\omega'\in \overline M$ such that $\tau(\omega)\le\tau(\omega')$,
\begin{align*}
\left\vert F(x,\omega,m)-F(x,\omega',m)\right\vert
&\le \int_0^{\tau(\omega)}|f(x,\omega,m,s)-f(x,\omega',m,s)|\, ds
+|\tau(\omega')-\tau(\omega)|\, \Vert f(x,\cdot,m,\cdot)\Vert_\infty\\
&\le K'd(\omega,\omega')\left(\Vert\tau\Vert_\infty
[f(x,\cdot)]_{H(\mathcal M_m)}
+\Vert \tau\Vert_{Lip}\, \sup_{\mathcal M_m}|f(x,\cdot)|\right)\, ,
\end{align*}
where we used the fact that we can choose $\beta$ in the definition~\eqref{eqdynamicmetric} (or in the metric $d$ in~\cite{MP24discrete}) and $K'>0$ such that the euclidean/riemmannian distance between $\varphi_s(\omega)$ and $\varphi_s(\omega')$ is smaller than $K'd(\omega,\omega')$ and so that $\tau$ is Lipschitz with respect to the dynamical metric $d$.
The same holds true if we replace $(f,F)$ by $(D_1f,D_1F)$. Thus we have proved that
\[
\sup_{x\in\mathbb R^d}\sum_{m\in\mathbb Z}[\Vert F(x,\cdot,m)\Vert_{Lip}+\Vert D_1 F(x,\cdot,m)\Vert_{Lip}]<\infty\, .
\]
Therefore \cite[Theorem 3.4]{MP24discrete} applies both for $F$ and for $D_1F$ ensuring Assumption (b). Theorem~\ref{thmExempleprecis}
thus follows by Corollary~\ref{corbillard}, as announced.
\end{proof}
\subsection{Proof of Corollary \ref{corbillard}}\label{sec:proofcoro}
First, notice that the fact that Assumption~(c) of Theorem~\ref{thm:mainflow} holds true with 
$\overline\tau:=\int_{\Mp}\tau\, d\mup$ follows from  the ergodicity of $(\Mp,\mup,\tp)$ (via the Kac Lemma).\\

To identify the limit in Corollary \ref{corbillard} from the one coming from Theorem \ref{thm:mainflow} we will use the following known result (see \cite[Proposition 1.5.3]{MPphd} for example).
\begin{lem}\label{lemtemplocvar}
The local time $(\tilde L_t(.))_{t \geq 0}$ of $(\tilde B_t=B'_{t/\overline\tau})_{t\ge 0}$ is given by
$\tilde L_t(x)=\overline\tau L'_{t/\overline\tau}(x)$, 
where 
$(L'_t)_{t\geq 0}$ is the local time of the Brownian motion $\left(B_t')
\right)_{t\geq 0}$.
\end{lem}
\begin{proof}[Proof of Corollary \ref{corbillard}]
From Theorem \ref{thm:mainflow} applied to the assumptions of Corollary \ref{corbillard}, we already know that the process $\left(E_t^\epsilon(x,\cdot)\right)_{t\geq 0}$ converges strongly in distribution with respect to $\nu$ and
for the local topology of $\mathcal C([0,+\infty))$ to some process $(Y_t(x,\cdot))_{t\geq 0}$ given by
\begin{align}\label{eqthmpert}
Y_t(x,\omega)&= \int_0^{t/\overline\tau} \sqrt{a(W_s(x))}dB_{L_s'(0)}(\omega)\\
& +\int_0^t D\overline f(W_s(x)) \exp\left(\int_{s}^{t}D\overline f(W_u(x))du\right)\int_0^{s/\overline\tau} \sqrt{a(W_u(x))}dB_{L_u'(0)}(\omega)ds\, .
\end{align}
Recall that $w_t(x)$ has been defined in~\eqref{EDflowW2} and satisfies
$w_t=W_{\overline \tau t}$. 
The conclusion of Corollary \ref{corbillard} is thus a direct consequence of theorem \ref{thm:mainflow} with a slight reformulation of 
$(v_{\frac{t}{\bar \tau}})_{t\geq 0}$ mentioned there :
\begin{align*}
   (v_{\frac{t}{\bar \tau}})_{t\geq 0}&
   = \int_0^{t/\overline\tau} \sqrt{a(w_s(x))}dB_{L_s'(0)}\\
   &
   = \int_0^{t/\overline\tau} \sqrt{a(w_{\overline\tau t}(x))}dB_{L_s'(0)}\\
   &
   =\int_0^t\sqrt{a(W_s(x))}dB_{L'_{s/\overline \tau}(0)}\\
   &
   =\int_0^t\sqrt{\tilde a(W_s(x))}dB''_{\tilde L_s(0)},
\end{align*}
where we set $B''_t=\sqrt{\overline\tau}B_{t/\overline\tau }$, since
$\tilde a=\frac{a}{\overline \tau}$ and  $
\tilde L_t(0)
=
\overline \tau L'_{\frac{t}{\overline \tau}}(0)
$. Furthermore $B''$ has the same distribution as $B$. 
Thus we only need to check that the law of those two match relations \eqref{eqbrowniencontin} and \eqref{eqtilda}.\\
Both these convergences are consequence of \cite[Theorem 3.9, and section 14]{billingsley2} on the Slutzky lemma and on the random change of time. Indeed, using Assumption~(b) of Theorem~\ref{thm:mainflow}, the Birkhoff ergodic theorem, the definition of $B'$ and~\eqref{TCLinf}, we know that
\begin{align*}
\left(\epsilon n_{\frac t \epsilon}\right)_{t> 0} \xrightarrow[\epsilon \rightarrow 0]{\nu-a.e} \left(\frac t {\overline{\tau}}\right)_{t\geq 0} \textit{ and } \left(
\epsilon^{\frac 12}S_{\lfloor t/\epsilon\rfloor}\Phi\right)_{t\geq 0} \xrightarrow[\epsilon \rightarrow 0]{\mathcal{L}_{\nu},\Vert\cdot\Vert_\infty}\left(B_t'\right)_{t\geq 0}\, ,
\end{align*}
and
\begin{align}
\left(\epsilon n_{\frac t \epsilon}\right)_{t\geq 0} \xrightarrow[\epsilon \rightarrow 0]{\nu-a.e} \frac t {\overline{\tau}} \textit{ and } \left(Z^\epsilon_t\right)_{t\geq 0}  \xrightarrow[\epsilon \rightarrow 0]{\mathcal{L}_{\nu},\Vert\cdot\Vert_\infty}\left(\sqrt{a(x)}B_{L'_t(0)}\right)_{t\geq 0}\, ,
\end{align}
with $Z^\epsilon_t:=\epsilon^{1/4}\sum_{k=0}^{\lfloor t/\epsilon\rfloor-1}F(x,T^k(.))$.
Therefore, it follows from~\cite[Theorem 3.9, and Section 14]{billingsley2}  that
\[\left(\epsilon^{1/2}S_{n_\frac{t}{\epsilon}}=\epsilon^{1/2}S_{\lfloor (\epsilon n_\frac{t}{\epsilon})/\epsilon}\rfloor\right)_{t\geq 0}
 \xrightarrow[\epsilon \rightarrow 0]{\mathcal{L}_{\nu},\Vert\cdot\Vert_\infty}\left(\tilde B_t=B_{t/\overline\tau}'\right)_{t\geq 0},
 \]
which ends the proof of~\eqref{eqbrowniencontin}; and we also obtain that
\[\left(Z^\epsilon_{\epsilon n_{t/\epsilon}}\right)_{t\geq 0}  \xrightarrow[\epsilon \rightarrow 0]{\mathcal{L}_{\nu},\Vert\cdot\Vert_\infty}\left(\sqrt{a(x)}B_{L'_{t/\overline\tau}(0)}=\sqrt{\tilde a(x)}B''_{\tilde L_{t}(0)}\right)_{t\geq 0}.   
\]
Since
\[
    \left|\epsilon^{1/4}\int_0^{t/\epsilon}f(x,\f_s(.))\, ds
    - Z^\epsilon_{\epsilon n_{t/\epsilon}}\right|
   \le 2\epsilon^{1/4}\Vert \tau\Vert_\infty\Vert f\Vert_\infty\, ,
\]
which goes to 0 as $\epsilon\rightarrow 0$, we conclude~\eqref{eqtilda}.

\end{proof}

\section{Averaging for suspension flow: Proof of Theorem~\ref{thm:mainflow}}\label{secproof}
In all this section, we assume the assumptions of Theorem~\ref{thm:mainflow}. 
Let $S>0$. 
We start by recalling a useful fact (see \cite[Section 7]{billingsley2}). We write $S':=\frac{S}{\inf\tau
}$ and $\boldsymbol{\omega}_{[0;S']}$ for the continuity modulus on the time interval $[0;S']$. 
\begin{fact}\label{continuitymodulus}
If $(\gamma^\epsilon)_\epsilon$ is a family of random processes 
tight with respect to $\overline\mu$ in $\mathcal C([0;S'], \mathbb R^d)$ as $\epsilon\rightarrow 0$, then, for any $\eta_0>0$ and $\vartheta_0>0$, there exist $\eta>0$ and $\epsilon_1>0$ such that
\begin{equation*}
\forall \epsilon\in(0,\epsilon_1],\quad
\overline\mu(\boldsymbol{\omega}_{[0;S']}(\gamma^\epsilon,\eta)>\vartheta_0)<\eta_0\, .
\end{equation*}
\end{fact}

A natural strategy to prove Theorem~\ref{thm:mainflow} is to approximate $X_t^\epsilon$ and $W_t$ using solutions $\tilde x_t^\epsilon$ and $\tilde w_t^\epsilon$ of the following differential equations defined on $\mathbb R^d\times M$
\begin{align*}
\frac{d\tilde x_t^\epsilon}{dt} (x,\omega)=F(\tilde x_t^\epsilon(x,\omega),T^{\lfloor \frac t \epsilon \rfloor}\omega)+\tau\circ T^{\lfloor \frac t \epsilon \rfloor}(\omega)\overline{f}(\tilde x_t^\epsilon(x,\omega))\, ,\quad \tilde x_0^\epsilon(x,\omega)=x\, ,
\end{align*}
and
\begin{align*}
\frac{d\tilde w_t^\epsilon}{dt} (x,\omega)=\tau\circ T^{\lfloor \frac t \epsilon \rfloor}(\omega)\overline{f}(\tilde w_t^\epsilon(x,\omega))\, ,\quad \tilde w_0^\epsilon(x,\omega)=x\, ,
\end{align*}
for $\omega \in M$. 
Notice for the last approximation that $\tilde w_t^\epsilon(x,\omega,n)$ with $(x,\omega,n)\in\mathbb R^d\times\overline M\otimes \mathbb Z$ does not depend on $n\in\mathbb Z$.
Thus $\tilde w_t^\epsilon$ admits a version on $(\Mp,\tp,\mup)$ which is a perturbed equation with averaged solution $w_{t}$
defined in~\eqref{EDflowW2}. 
Due to an argument by Kifer in~\cite{kifer} (see also \cite[prop. 3.2.2]{peneESAIMPS}), it follows from~\eqref{boundtau} 
and from the fact that $f$ and $\overline f$ are uniformly bounded and uniformly Lipschitz in the first variable $x\in\mathbb R^d$
that there exists a constant $C>0$ such that
for all $(\omega,s) \in \M$,
\begin{align}\label{eqxtild}
\sup_{x\in \mathbb{R}^d}\sup_{t\in [0,T]}\sup_{(\omega,l)\in M}\sup_{s\in[0,\tau(\omega))}|X_t^\epsilon(x,\omega,l,s)-\tilde x_{\epsilon n_{t/\epsilon}(\omega)}^\epsilon(x,\omega,l)|\leq C\epsilon\, ,
\end{align}
\begin{align}\label{eqwtilde}
\sup_{x\in \mathbb{R}^d}\sup_{t\in [0,T]}\sup_{\omega\in \overline M}\sup_{s\in[0,\tau(\omega))}|W_t(x)-\tilde w_{\epsilon n_{t/\epsilon}(\omega)}^\epsilon(x,\omega)|\leq C\epsilon\, ,
\end{align}
where $n_t(\omega)$ is the number of visits to $M\times 0$
of the orbit $(\varphi_s(\omega,0))_{s\in]0;t]}$, this corresponds to
\[
n_t(\omega):=\max\left\{n\ge 0\, :\, \sum_{k=0}^{n-1}\tau\circ
T^k(\omega)\le t\right\}\, .
\]
Thus, up to a change of time, we solely have to estimate the error made between $(\tilde x_t^\epsilon)_t$ and $(\tilde w_t^\epsilon)_t$:
\begin{align}
\tilde{e}_t^\epsilon:=&\tilde x_t^\epsilon -\tilde w_t^\epsilon = 
(\epsilon/\mathfrak a_\epsilon)v_t^\epsilon+c^\epsilon_t(x,\omega) \label{eqbirkper}\\ &+ 
\int_0^t F(\tilde x_s^\epsilon,T^{\lfloor s /\epsilon \rfloor})-F(\tilde w_s^\epsilon,T^{\lfloor s /\epsilon \rfloor})ds+\int_0^t\tau\circ T^{\lfloor s/\epsilon \rfloor} \left( \overline f(\tilde x_s^\epsilon)-\overline f (\tilde w_s^\epsilon)\right)ds\, ,\nonumber
\end{align}
with
\[
c^\epsilon_t(x,\omega):=\int_0^t \left(F(\tilde w_s^\epsilon,T^{\lfloor s /\epsilon \rfloor})-F(w_s,T^{\lfloor s /\epsilon \rfloor})\right)\, ds 
\]
and recalling that we have set 
\[
v_t^\epsilon(x,\omega)=
(\mathfrak a_\epsilon/\epsilon)\int_0^t F( w_s(x),T^{\lfloor s /\epsilon \rfloor}(\omega))\, ds\, ,
\]
in Assumption~(b) of Theorem~\ref{thm:mainflow}. 
%
Let us study the right hand side of \eqref{eqbirkper}.
By hypothesis, we already know that 
the family of processes $((v_t^\epsilon)_{t\geq 0})_{\epsilon >0}$ converges strongly in distribution with respect to $\mu$ and for the uniform metric on $[0,S']$. 
We will prove the next lemma.
\begin{lem}\label{controleGeps}
Let $x\in\mathbb R$. The family of random variables 
\begin{equation}
\left(
(\mathfrak a_\epsilon/\epsilon)\sup_{t\in[0;S']}\left\vert c_t(x,\cdot)\right\vert\right)_{\epsilon>0}
\end{equation}
converges to 0 in probability (with respect to $\overline \mu$), as $\epsilon\rightarrow 0$.
\end{lem}
To this end, we start by proving the following easy estimate.
\begin{lem}\label{wtilde-w}
Let $x\in\mathbb R$. 
The family of random variables
$
\sqrt{\mathfrak a_\epsilon/\epsilon}\sup_{t\in[0;S']}|\widetilde w_s^\epsilon(x,\cdot)-w_s(x)|
$ converges in probability to 0
(with respect to $\overline\mu$), as $\epsilon\rightarrow 0$.
\end{lem}
\begin{proof}
It follows from the differential equations satisfied respectively by
$\widetilde w^\epsilon$ and $w$ that
\begin{align*}
|\widetilde w_t^\epsilon(x,\cdot)-w_t(x)|
&\le \left|\int_0^t(\tau\circ T^{\lfloor s/\epsilon\rfloor}-\overline\tau)\overline f(w_s(x))\, ds\right|\\
&\quad\quad +\int_0^t\max\tau \Vert D_1\overline f\Vert_\infty |\widetilde w_s^\epsilon(x,\cdot)-w_s(x)|\, ds\, .
\end{align*}
It then follows from the Gr\"onwall Lemma that there exists $C_0>0$ such that, for all $(x,\omega)\in \mathbb R\times \overline M$,
\begin{align}
\sup_{t\in[0;S']}|\widetilde w_t^\epsilon(x,\omega)-w_t(x)|
&\le C_0 \sup_{t\in[0;S']}\left|\int_0^t(\tau\circ T^{\lfloor s/\epsilon\rfloor}-\overline\tau)\overline f(w_s(x))\, ds\right|\\
&\le C_0 \sqrt{
\epsilon/\mathfrak a_\epsilon}\sup_{t\in[0;S']}\left|z_t^\epsilon
\right|\, ,
\end{align}
which, combined with Assumption~(d) of Theorem~\ref{thm:mainflow}, proves the lemma, once multiplied by $\sqrt{\mathfrak a_\epsilon/\epsilon}$, as $\epsilon\rightarrow 0$.
\end{proof}

\begin{proof}[Proof of Lemma~\ref{controleGeps}]
Let $t\in[0;S']$. 
Since $F$ is twice differentiable in its first variable $x\in\mathbb R$ with uniformly bounded second order derivative, it follows that
\begin{align}\label{eqamajorer}
\left|
c_t^\epsilon(x,\omega)-\int_0^tD_1F(w_s,T^{\lfloor s/\epsilon \rfloor})(\tilde w^\epsilon_s-w_s)ds\right|
\le \frac {t\Vert D_1^2F\Vert_\infty}2 \sup_{s\in[0;S']}\left|\tilde w^\epsilon_s-w_s\right|^2  \, ,
\end{align} 
and, it follows from Lemma~\ref{wtilde-w}
that this quantity, taken at $(x,\cdot)$ and multiplied by 
$\mathfrak a_\epsilon/\epsilon$, vanishes in probability as $\epsilon\rightarrow 0$.
Setting
$$
G_t^\epsilon(x,\omega):=
(\mathfrak a_\epsilon/\epsilon)\int_0^tD_1F(w_s(x),T^{\lfloor s/\epsilon \rfloor}(\omega))(\tilde w^\epsilon_s(x,\omega)-w_s(x))ds
$$
it remains to prove that 
\begin{equation}\label{cvgceGproba}
\sup_{[0;S']}|G^\epsilon(x,\cdot)|\rightarrow 0\mbox{  in probability as } \epsilon\rightarrow 0\, .
\end{equation}
Let $\eta>0$. Observe that
\begin{align}
G^\epsilon
&=
(\mathfrak a_\epsilon/\epsilon)\int_0^tD_1F(w_s
,T^{\lfloor s/\epsilon \rfloor})(\tilde w^\epsilon_s-w_s)ds \nonumber\\
&=
(\mathfrak a_\epsilon/\epsilon)\int_0^t\sum_{k=0}^{\lfloor t/\eta\rfloor}D_1F(w_s,T^{\lfloor s/\epsilon \rfloor})\left(\tilde w^\epsilon_{k\eta}- w_{k\eta}\right.\\
&\left. 
\quad +\int_{k\eta}^s
\left(
\tau\circ T^{\lfloor u/\epsilon\rfloor}
\overline f(\tilde w_u^\epsilon)
-\overline \tau
\overline f(w_u)
\right)
du\right)1_{
 k\eta \leq s < (k+1)\eta} \, ds\nonumber\\
&=
(\mathfrak a_\epsilon/\epsilon)\sum_{k=0}^{\lfloor t/\eta
\rfloor}\left( \int_{
k\eta}^{\min(t,
(k+1)\eta)}D_1F(w_s
,T^{\lfloor s/\epsilon \rfloor})ds(\tilde w^\epsilon_{
k\eta}-w_{
k\eta})\right. \label{fubini}\\
&\left.+\int_{
k\eta}^{\min(t,
(k+1)\eta)}\left(\int_{u}^{\min(t,
(k+1)\eta)} D_1F(w_s
,T^{\lfloor s/\epsilon \rfloor})\, ds\right)\left(\tau\circ T^{\lfloor u/\epsilon\rfloor}
\overline f(\tilde w_u^\epsilon)-
\overline \tau
\overline f(w_u
)\right)
du\right),\nonumber
\end{align}
where we used Fubini relation at equation \eqref{fubini}. Recalling that, for any $\epsilon >0$, the process $(\dot{v}_t^\epsilon)_{t\geq 0}$ is given by $\dot{v}_t^\epsilon=
(\mathfrak a_\epsilon/\epsilon)
\int_0^tD_1F(w_s,T^{\lfloor s/\epsilon \rfloor})ds$, it follows that
\begin{align}
\sup_{t\in[0;S']}&|G^\epsilon_t(x,\cdot)|\leq 
K'_\epsilon(x,\cdot)
\boldsymbol{\omega}_{[0,S']}(\dot v^\epsilon(x,\cdot),
\eta)\, ,\label{eqintvepsilon}
\end{align}
with $K'_\epsilon:=S'\left(
\eta^{-1}
\sup_{[0;S']}|\tilde w^\epsilon- w|+2\|\tau\|_\infty\|\overline f\|_\infty
\right)$. 
To conclude, given $\eta_0>0$, we use
Assumption~(b) of Theorem~\ref{thm:mainflow} and Fact~\ref{continuitymodulus} to fix $\eta>0$ and $\epsilon_1>0$ so that 
\begin{equation}\label{tight}
    \forall \epsilon\in(0;\epsilon_1],\quad \overline\mu\left(S'(
1+2\Vert\tau\Vert_\infty\Vert\overline f\Vert_\infty)\boldsymbol{\omega}_{[0,S']}(\dot v^\epsilon(x,\cdot),
\eta)>\frac{\eta_0}2\right)<\frac{\eta_0}2\, 
\end{equation}
and then we use Lemma~\ref{wtilde-w}
which ensures that
there exists $\epsilon_0\in(0,\epsilon_1)$ such that
\[\forall \epsilon\in(0,\epsilon_0],\quad
\overline\mu\left(\eta^{-1}\sup_{[0,S']}|\widetilde w^\epsilon-w|>1\right)<\frac{\eta_0}2\, ,
\]
which combined with~\eqref{tight}
ensures that 
\[
\forall \epsilon\in(0,\epsilon_0],\quad 
\overline\mu\left(\sup_{[0;S']}|G^\epsilon(x,\cdot)|>\eta_0\right)<\eta_0\, ,
\]
ending the proof of~\eqref{cvgceGproba} and so of Lemma~\ref{controleGeps}.
\end{proof}

Thus, it follows from Lemma~\ref{controleGeps} that Equation~\eqref{eqbirkper} becomes
\[
(\mathfrak a_\epsilon/\epsilon)
\tilde e_t^\epsilon= v_t^\epsilon+ o(1)
+ 
\int_0^t F(\tilde x_s^\epsilon,T^{\lfloor s /\epsilon \rfloor})-F(\tilde w_s^\epsilon,T^{\lfloor s /\epsilon \rfloor})ds+\int_0^t\tau\circ T^{\lfloor s/\epsilon \rfloor} \left( \overline f(\tilde x_s^\epsilon)-\overline f (\tilde w_s^\epsilon)\right)ds\, ,
\]
for the convergence in probability (with respect to $\overline\mu$) and for the uniform norm on $\mathcal C([0;S'])$.
We will see in Lemma~\ref{lemmajoreteps} that $
(\mathfrak a_\epsilon/\epsilon)
\tilde e_t^\epsilon$ is well approximated by the process $y_t^\epsilon$ 
studied in the following result.
\begin{prop}\label{propconvyeps}
For every $x\in\mathbb R^d$, 
the family of processes $(\tilde y^\epsilon(x,\cdot))_{\epsilon>0}$ defined by 
$$
\tilde y_t^\epsilon(x,\cdot)= v_t^\epsilon(x,\cdot) +\int_0^t \tau \circ T^{\lfloor s/\epsilon \rfloor} D \overline f (w_s(x)) \tilde y_s^\epsilon(x,\cdot)\,  ds
$$
 converges 
 in distribution on 
 with respect to 
 $\mup$
 and for the uniform metric on the space $\mathcal C([0,S'],\mathbb{R}^d)$ to the process $(y_t(x,\cdot))_{t \geq 0}$
\begin{align*}
y_t(x,\cdot)=v_t(x,\cdot)+\int_0^t \overline \tau D \overline f (w_s(x)) _s y_s(x,\cdot)\, ds\, .
\end{align*}
\end{prop}
\begin{proof}
Fix $\epsilon>0$ and denote $\mathcal{F}_\epsilon: \mathcal C ([0,S'],\mathbb{R}^d)\rightarrow \mathcal C ([0,S'],\mathbb{R}^d)$, the application such that $y:=\mathcal{F}_\epsilon(z)$ where $(y_t)_{t\geq 0}$ is defined, for $z\in \mathcal C ([0,S'],\mathbb{R}^d)$, as the solution of the following variational equation 
\begin{align*}
\left(\mathcal{F}_\epsilon(z)\right)_t=z_t+\int_0^t \tau \circ T^{\lfloor s/\epsilon \rfloor} D\overline{f}(w_s(x))\left(\mathcal {F}_\epsilon(z)\right)_s\, ds\, ,
\end{align*}
and analogously, we define $\mathcal{F}_0: \mathcal C ([0,S'],\mathbb{R}^d)\rightarrow \mathcal C ([0,S'],\mathbb{R}^d)$ with the following variational equation
\begin{align*}
\left(\mathcal{F}_0(z)\right)_t=z_t+\int_0^t \overline \tau D\overline{f}(w_s(x))\left(\mathcal {F}_0(z)\right)_sds.
\end{align*}
Observe that
\[
(\mathcal F_\epsilon(z))_t=z_t+\int_0^tK_\epsilon(s,t)z_s\, ds\quad\mbox{and}\quad
(\mathcal F_0(z))_t=z_t+\int_0^tK_0(s,t)z_s\, ds\, ,
\]
with 
\[K_\epsilon(s,t):=\tau\circ T^{\lfloor \frac s\epsilon\rfloor}D\overline f(w_s)e^{\int_s^t\tau\circ T^{\lfloor \frac u\epsilon\rfloor}D\overline f(w_u)\, du}\]
and $K_0(s,t):=\overline\tau D\overline f(w_s)e^{\int_s^t\overline\tau D\overline f(w_u)\, du}$.

In particular $\mathcal F_\epsilon$ and $\mathcal F_0$ are uniformly Lipschitz with respect to the uniform metric. 
This, combined with 
the convergence of $v^\epsilon$ to $v$,
ensures that $\mathcal F_0(v^\epsilon)$ converges strongly in distribution
with respect to $\mu$ to $\mathcal F_0(v)$.
It remains to prove that $\mathcal F_\epsilon(v^\epsilon)-\mathcal F_0(v^\epsilon)$ converges in probability to 0
as $\epsilon\rightarrow 0$.
To this end, we write
\begin{align*}
&\left|(\mathcal F_\epsilon(v^\epsilon(x,\cdot)))_t-(\mathcal F_0(v^\epsilon(x,\cdot)))_t\right|\\
&=\left| \int_0^tD\overline{f}(w_s(x))\left(\tau\circ T^{\lfloor s/\epsilon\rfloor}(\mathcal F_\epsilon(v^\epsilon(x,\cdot)))_s-\overline\tau
(\mathcal F_0(v^\epsilon(x,\cdot)))_s\right)\, ds
\right|\\
&\le \left| \int_0^t\tau\circ T^{\lfloor s/\epsilon\rfloor} D\overline{f}(w_s(x))\left((\mathcal F_\epsilon(v^\epsilon(x,\cdot)))_s-
(\mathcal F_0(v^\epsilon(x,\cdot)))_s\right)\, ds\right|\\
&+\left| \int_0^t(\tau\circ T^{\lfloor s/\epsilon\rfloor}-\bar\tau)D\overline{f}(w_s(x))
(\mathcal F_0(v^\epsilon(x,\cdot)))_s\, ds
\right|\\
&\le A_t^\epsilon(x,\cdot)+\int_0^t\Vert\tau\Vert_\infty
\Vert D\overline f\Vert_\infty \left|(\mathcal F_\epsilon(v^\epsilon(x,\cdot)))_s-
(\mathcal F_0(v^\epsilon(x,\cdot)))_s\right|\, ds\, ,
\end{align*}
with 
\[
A_t^\epsilon(x,\cdot):=\left| \int_0^t (\tau\circ T^{\lfloor s/\epsilon\rfloor}-\bar\tau)D\overline{f}(w_s(x))
(\mathcal F_0(v^\epsilon(x,\cdot)))_s\, ds
\right|\, .
\]
It then follows from the Gr\"onwall lemma that
\begin{equation}\label{Feps-F0}
\sup_{t\in[0,S']}\left|(\mathcal F_\epsilon(v^\epsilon(x,\cdot)))_t-(\mathcal F_0(v^\epsilon(x,\cdot)))_t\right|\le
\sup_{u\in[0,S']}|A_u^\epsilon(x)|e^{S'\Vert\tau\Vert_\infty\Vert D\overline f\Vert_\infty}\, .
\end{equation}
Since $\mathcal F_0$ is Lipschitz, noting $L_0$ for its Lipschitz constant 
\[
\boldsymbol{\omega}(\mathcal F_0(v^\epsilon(x,\cdot))\le L_0\, \boldsymbol{\omega}(v^\epsilon(x,\cdot))\, .
\]
Furthermore
\begin{align*}
\sup_{t\in[0;S']}|A_t^\epsilon|&=\sup_{k=0}^{S'/\eta}|A_{k\eta}^\epsilon| + 2\eta \Vert D\overline f\Vert_\infty \Vert\tau\Vert_\infty \sup_{[0;S']}|\mathcal F_0(v^\epsilon)|\\
&\leq \sum_{k=0}^{S'/\eta}\left|\int_{k\eta}^{(k+1)\eta}(\tau\circ T^{\lfloor s/\epsilon\rfloor}-\overline\tau)
D\overline f(w_s) (\mathcal F_0(v^\epsilon))_{k\eta}\, ds\right|\\
&+ 2 \Vert D\overline f\Vert_\infty \Vert\tau\Vert_\infty \left(\eta\sup_{t\in[0;S']}|(\mathcal F_0(v^\epsilon))_t|+ S' \boldsymbol{\omega}_{[0;S']}(\mathcal F_0(v^\epsilon(x,\cdot)),\eta)\right)\, .
\end{align*}
Let $x\in\mathbb R^d$ and $\eta_0>0$. Assumption~(b) of Theorem~\ref{thm:mainflow} ensures the  convergence in distribution of $(v^\epsilon(x,\cdot))_\epsilon$ and so of $(\mathcal F_0(v^\epsilon(x,\cdot)))_\epsilon$, and from 
Fact~\ref{continuitymodulus} that
we can fix $\eta>0$ and $\epsilon_1>0$ so that, for every $\epsilon\in(0,\epsilon_1)$,
\[
\overline\mu \left(
2 \Vert D\overline f\Vert_\infty \Vert\tau\Vert_\infty \left(\eta\sup_{[0;S']}|\mathcal F_0(v^\epsilon(x,\cdot))|+ S'\boldsymbol{\omega}_{[0;S']}(\mathcal F_0(v^\epsilon(x,\cdot)),\eta)\right)>\frac{\eta_0}2
\right)<\frac{\eta_0}2\, .
\]
Then, we fix $\epsilon_0\in(0,\epsilon_1)$ small enough so that,
for all $\epsilon\in(0;\epsilon_0]$,
\[
\overline\mu\left(
\sum_{k=0}^{S'/\eta}\left|\int_{k\eta}^{(k+1)\eta} (\tau\circ T^{\lfloor s/\epsilon\rfloor}-\overline\tau)
D\overline f(w_s(x)) (\mathcal F_0(v^\epsilon(x,\cdot)))_{k\eta}\, ds\right|>\frac{\eta_0}2
\right)<\frac{\eta_0}2\, ,
\]
using the convergence in distribution to 0, with respect to $\overline\mu$ as $\epsilon\rightarrow 0$, of $(\dot z_{k\eta}^\epsilon)_{\epsilon>0}$, for all $k\in\mathbb N$ (this comes from Assumption~(d) of Theorem~\ref{thm:mainflow}).
Therefore 
\[
\forall \epsilon\in(0;\epsilon_0],\quad 
\overline\mu\left(\sup_{t\in[0;S']}|A_t^\epsilon(x,\cdot)|>\eta_0 \right)<\eta_0\, ,
\]
and so 
$\sup_{t\in[0;S']}|A_t^\epsilon(x,\cdot)|$
converges in probability (with respect to $\overline\mu$) to 0. This, combined with~\eqref{Feps-F0} implies that
$\mathcal F_\epsilon(v^\epsilon(x,\cdot))-\mathcal F_0(v^\epsilon(x,\cdot))$ converges in probability to 0
as $\epsilon\rightarrow 0$, ending the proof of the proposition, since $\mathcal F_0(v^\epsilon(x,\cdot))$
converges in distribution with respect to $\overline\mu$ to $\mathcal F_0(v(x,\cdot))$.
 \end{proof}
\begin{lem}\label{lemmajoreteps}
For every $x\in\mathbb R^d$, 
the sequence $(
\sup_{t \in [0,S']}|
(\mathfrak a_\epsilon/\epsilon)
\tilde e_t^\epsilon(x,.)-\tilde y_t^\epsilon(x,.)|)_{\varepsilon>0}
$
converges in probability (with respect to $\overline\mu$) to 0,
as $\varepsilon\rightarrow 0$.
\end{lem}

\begin{proof}
As in the proof for the discrete time dynamics (see \cite[Lemma 7.2]{MP24discrete}), we introduce the following quantities
\begin{align*}
a^\epsilon_t(x,\omega)&:=
(\mathfrak a_\epsilon/\epsilon)\tilde e_t^\epsilon(x,\omega)- v_t^\epsilon(x,\omega)\\
&-
(\mathfrak a_\epsilon/\epsilon)\int_0^t\left(D_1F(w_s(x),T^{\lfloor \frac s \epsilon\rfloor}(\omega))+\tau\circ T^{\lfloor s/\epsilon\rfloor}(\omega)D\overline f(w_s(x))\right)\tilde e_s^\epsilon(x,\omega)ds.\\
\end{align*}
Using the definition of $\tilde e_t^\epsilon$ (see equation \eqref{eqbirkper}) and the Taylor expansion, we obtain,
\begin{align}
|a^\epsilon_t
|&\le 
(\mathfrak a_\epsilon/\epsilon)\left|c^\epsilon_t
\right|
+
(\mathfrak a_\epsilon/\epsilon)\left|\int_0^t\left(F\left(\tilde x_s^\epsilon
,T^{\lfloor \frac s \epsilon \rfloor}
\right)-F\left(\tilde w_s^\epsilon
,T^{\lfloor \frac s \epsilon \rfloor}
\right) \right.\right.\nonumber\\
&+\tau\circ T^{\lfloor s/\epsilon\rfloor}
\left(\overline f(\tilde x_s^\epsilon
)-\overline f(\tilde w_s^\epsilon
)\right)\nonumber\\
&\left. -\left(D_1F(w_s
,T^{\lfloor \frac s \epsilon\rfloor}
)+\tau\circ T^{\lfloor s/\epsilon\rfloor}
D\overline f(w_s
)\right)\tilde e_s^\epsilon
\, ds\right|\nonumber\\
&\leq 
(\mathfrak a_\epsilon/\epsilon)|c^\epsilon_t
|\label{eqrougeajout}\\
&+ 
(\mathfrak a_\epsilon/\epsilon)\int_0^T\left|\|D_1^2F\|_{\infty}+\|\tau\|_\infty\|D_1^2\overline f\|_\infty)\right|\left(\tilde e_s^\epsilon
\right)^2\, ds\label{eqconva0}\\
&+
(\mathfrak a_\epsilon/\epsilon)\left|\int_0^t \tau\circ T^{\lfloor s/\epsilon\rfloor}
\left(D\overline f(\tilde w_s^\epsilon)-D\overline f(w_s)\right)\tilde e_s^\epsilon ds\right|\label{eqconva}\\
&+
(\mathfrak a_\epsilon/\epsilon)\left|\int_0^t \left(D_1 F(\tilde w_s^\epsilon,T^{\lfloor \frac s \epsilon \rfloor}
)-D_1F(w_s,T^{\lfloor \frac s \epsilon \rfloor}
)\right)\tilde e_s^\epsilon ds\right|\, .\label{eqholder2}
\end{align}
Lemma~\ref{controleGeps} ensures that the supremum over $t\in[0;S']$ of \eqref{eqrougeajout} converges in probability with respect to $\overline\mu$ to $0$.
Terms \eqref{eqconva} and \eqref{eqholder2} can be treated together as follows. For any $t\in[0;S']$, 
\begin{align}
\eqref{eqconva}+\eqref{eqholder2}
&\leq S'\left(\|\tau\|_\infty\|D^2\overline f\|_\infty
   +\Vert D_1^2F\Vert_\infty\right)\sup_{s\in[0;S']}|\tilde w_s^\epsilon -w_s|
\sup_{t \in [0,S']}\left|\frac{\tilde e_t^\epsilon(x,.)}{\epsilon
/\mathfrak a_\epsilon}\right|\, .\label{AAAA1}
\end{align}
But, it follows from~\eqref{eqbirkper} that
\[
(\mathfrak a_\epsilon/\epsilon)|\widetilde e_t^\epsilon|\le \left|v_t^\epsilon+
(\mathfrak a_\epsilon/\epsilon)c^\epsilon_t(x,\omega)\right|+
(\mathfrak a_\epsilon/\epsilon)\int_0^t \left(\Vert D_1F\Vert_\infty+\Vert\tau\Vert_\infty\Vert D\overline f\Vert_\infty\right)|\widetilde e_s^\epsilon|\, ds\, ,
\]
which combined with the Gr\"onwall lemma leads to
\begin{equation}\label{majosupeeps}
\sup_{t\in[0;S']}
(\mathfrak a_\epsilon/\epsilon)|\widetilde e_t^\epsilon|\le \sup_{t\in[0;S']}(|v_t^\epsilon|+
(\mathfrak a_\epsilon/\epsilon)
|c^\epsilon_t(x,\omega)|)e^{
S'\left(\Vert D_1F\Vert_\infty+\Vert\tau\Vert_\infty\Vert D\overline f\Vert_\infty\right)}\, .
\end{equation}
Thus, according to Lemmas~\ref{controleGeps} and~\ref{wtilde-w}, and using the fact that $\epsilon=\mathcal O(\mathfrak a_\epsilon)$ and to the convergence of $v^\epsilon$, it follows from~\eqref{AAAA1} and~\eqref{majosupeeps} that the supremum over $t\in[0;S']$ of the sum of the terms~\eqref{eqconva} and~\eqref{eqholder2} converges in distribution to $0$.\\
Inequality~\eqref{majosupeeps} also implies that
the supremum over $t\in[0;S']$ of the term~\eqref{eqconva0} also goes to $0$.
We conclude that
\begin{align}\label{bornep}
\sup_{t \in [0,S']}&\left|a_t^\epsilon(x,
\cdot
)\right|\rightarrow 0,\quad\mbox{in probability with respect to }\overline\mu\, .
\end{align}

Define
\begin{align*}
&b_t^\epsilon(x,\omega):=\tilde y_t^\epsilon(x,\omega)- v_t^\epsilon(x,\omega)-\int_0^t \left(D_1 F(w_s(x),T^{\lfloor \frac s \epsilon \rfloor}(\omega))+\tau\circ T^{\lfloor s/\epsilon\rfloor}(\omega)D\overline f(w_s(x))\right)\widetilde y_s^\epsilon(x,\omega)\, ds\nonumber\\
&=-\int_0^t D_1 F(w_s(x),T^{\lfloor \frac s \epsilon \rfloor}(\omega))\widetilde y_s^\epsilon(x,\omega)\, ds\nonumber\, .
\end{align*}
Thus
\begin{align*}
&\sup_{t\in[0;S']}|b_t^\epsilon
|\le \sum_{k=0}^{S'/\eta}\left|\int_{k\eta}^{(k+1)\eta} D_1 F(w_s
,T^{\lfloor \frac s \epsilon \rfloor}
)\widetilde y_s^\epsilon
\, ds\right|+\eta\Vert D_1F\Vert_\infty\sup_{
[0;S']}|\widetilde y
^\epsilon
|\\
&\le \sum_{k=0}^{S'/\eta}\left|y_{k\eta}^\epsilon\int_{k\eta}^{(k+1)\eta} D_1 F(w_s
,T^{\lfloor \frac s \epsilon \rfloor}
)\, ds\right|+\Vert D_1F\Vert_\infty
\left(
\eta\sup_{s\in[0;S']}|\widetilde y_s^\epsilon
|+S'\boldsymbol{\omega}_{[0;S']}(\widetilde y^\epsilon
,\eta)
\right)\, .
\end{align*}
Thus, 
\begin{align*}
&\sup_{t\in[0;S']}|b_t^\epsilon
\le \sum_{k=0}^{S'/\eta}\left|y_{k\eta}^\epsilon(\epsilon/\mathfrak a_\epsilon)(\dot v_{(k+1)\eta}-\dot v_{k\eta})\right|+\Vert D_1F\Vert_\infty
\left(
\eta\sup_{s\in[0;S']}|\widetilde y_s^\epsilon
|+S'\boldsymbol{\omega}_{[0;S']}(\widetilde y^\epsilon
,\eta)
\right)\, ,
\end{align*}
where we used the definition of $\dot v$ given in Assumption~(b) of Theorem~\ref{thm:mainflow}.\\
Let $\eta_0>0$. Using the convergence in distribution of $(\widetilde y^\epsilon)_\epsilon$ from Proposition \ref{propconvyeps}
combined with Fact~\ref{continuitymodulus}, we fix $\eta>0$ and $\epsilon_1>0$ such that, for all $\epsilon\in(0;\epsilon_1]$,
\[
\overline\mu\left(
\Vert D_1F\Vert_\infty\left(\eta\sup_{s\in[0;S']}|\widetilde y_s^\epsilon(x,\cdot)|+S'\boldsymbol{\omega}_{[0;S']}(\widetilde y^\epsilon(x,\cdot),\eta)\right)>\frac{\eta_0}2
\right)<\frac{\eta_0}2\, .
\]
Finally, 
using the fact that $\epsilon=o(\mathfrak a_\epsilon)$, and that 
the continuous processes $(\widetilde y^\epsilon_s(x,\cdot))_{s\in[0;S']}$
and $(\dot v_t^\epsilon(x,\cdot)
)_{t\in[0;S']}$ are both 
tight, we fix 
$\epsilon_0\in(0,\epsilon_1)$ such that, for all $\epsilon\in(0;\epsilon_0]$, 
\[
\overline\mu\left(
\sum_{k=0}^{S'/\eta}\left|
\tilde y_{k\eta}^\epsilon(x,\cdot)
(\epsilon/\mathfrak a_\epsilon)(\dot v_{(k+1)\eta}-\dot v_{k\eta})\right|>\frac{\eta_0}2\right)<\frac{\eta_0}2\, .
\]
We conclude that
\begin{align}\label{borneppourb}
\sup_{t \in [0,S']}&\left|b_t^\epsilon(x,
\cdot
)\right|\rightarrow 0,\quad\mbox{in probability with respect to }\overline\mu\, .
\end{align}
Since 
\begin{align*}
&
(\mathfrak a_\epsilon/\epsilon)\tilde e_t^\epsilon(x,\omega)-\widetilde y_t^\epsilon(x,\omega)=a_t^\epsilon(x,\omega)-b_t^\epsilon(x,\omega)\\
&+\int_0^tD_1F(w_s(x),T^{\lfloor \frac s \epsilon \rfloor}(\omega)+\tau\circ T^{\lfloor s/\epsilon\rfloor}D\overline f(w_s(x)))(\epsilon^{-\frac 3 4}e_s^\epsilon(x,\omega)-\widetilde y_s^\epsilon(x,\omega))ds,
\end{align*}
we obtain using Grönwall lemma
\begin{align*}
\left|
(\mathfrak a_\epsilon/\epsilon)\tilde e_t^\epsilon(x,\omega)-y_t^\epsilon(x,\omega)\right| \leq \sup_{u\in [0,S']}\left|a_u^\epsilon(x,\omega)-b_u^\epsilon(x,\omega)\right|e^{S'\|D_1F\|_{\infty}+\|\tau\|_\infty\|D\overline f\|_\infty}.
\end{align*}
The result follows from this inequality combined with~\eqref{bornep} and~\eqref{borneppourb}.
\end{proof}
\begin{prop}
For every $x\in\mathbb R^d$, as $\varepsilon\rightarrow 0$
the process $((\mathfrak a_\epsilon/\epsilon)\tilde e_t^\epsilon(x,\cdot))_{t\in[0;S']}$ converges in distribution with respect to $\overline\mu$ in $\mathcal C([0,S'])$, to $g(y_t(x,\cdot))_{t\in[0;S']}$
defined in Proposition~\ref{propconvyeps}.
\end{prop}
\begin{proof}
This follows directly from Proposition~\ref{propconvyeps}, combined with Lemma~\ref{lemmajoreteps}.
\end{proof}
\begin{lem}\label{lemmup}
Fix $x\in\mathbb R^d$. 
The sequence $(((\mathfrak a_\epsilon/\epsilon)\tilde e_{\epsilon n_{t/\epsilon}}^\epsilon(x,\cdot))_{t\geq 0})_{\epsilon >0}$ in the space $D([0,S])$ converges strongly in distribution with respect to $\mu$
towards $(y_{t/\overline \tau}(x,\cdot))_{t\geq 0}$.
\end{lem}

\begin{proof}
We first prove the convergence of $((\mathfrak a_\epsilon/\epsilon)(\tilde e_{\epsilon n_{t/\epsilon}}^\epsilon(x,\cdot))_{t\geq 0})_{\epsilon >0}$ 
with respect to $\overline\mu$.
Assumption (c) of Theorem~\ref{thm:mainflow} ensures that $(\epsilon n_{\frac{t}{\epsilon}})_{t\in (0,S]}$ converges
in probability, with respect to $\mup$,
to $(\frac{t}{\overline \tau})_{t\in (0,S]}$. Since $(\tilde e_t^\epsilon(x,\cdot))_{t\geq 0}$ converges in distribution for $\mup$ in $D[0,\frac{S}{\overline \tau}]$ when $\epsilon$ tends to $0$. \cite[Theorem 3.9, Section 14]{billingsley2} applied to the couple $(((\mathfrak a_\epsilon/\epsilon)\tilde e_t^\epsilon(x,\cdot))_{t\geq 0}, \epsilon n_{\frac{t}{\epsilon}})$ ensures the convergence of $(((\mathfrak a_\epsilon/\epsilon)\tilde e_{\epsilon n_{t/\epsilon}}^\epsilon(x,\cdot))_{t> 0})_{\epsilon >0}$ to $(y_{t/\overline \tau}(x,\cdot))_{t> 0}$.\\

In order to apply Zweim\"uller's theorem \cite[Theorem 1.1]{zweimuller}, we check that $(\mathfrak a_\epsilon/\epsilon)\sup_{t\in[0,S]}\|\tilde e_{\epsilon n_{t/\epsilon}(\cdot)}^\epsilon(x,\cdot)-\tilde e_{\epsilon n_{t/\epsilon}(\cdot)}^\epsilon(x,T(\cdot))\|_\infty$ converges pointwise
to 0. To this end, it is enough to check that
\begin{align}\label{eqetzweim}
(\mathfrak a_\epsilon/\epsilon)\sup_{t\in[0,S']}|\tilde e_t^\epsilon(x,\omega)-\tilde e_t^\epsilon(x,T(\omega))|\xrightarrow[\epsilon \rightarrow 0]{} 0.
\end{align}

Notice that for any $(t,\omega)\in [0,S']\times M$, 
\begin{align}
|\tilde e_t^\epsilon(x,\omega)-\tilde e_t^\epsilon(x,T(\omega))|&\leq |\tilde x_t^\epsilon(x,\omega)-\tilde x_t^\epsilon(x,T(\omega))|\label{eqxgron}\\
&+|\tilde w_t^\epsilon(x,\omega)-\tilde w_t^\epsilon(x,T(\omega))|\label{eqwgron}. 
\end{align}
The terms \eqref{eqxgron} and \eqref{eqwgron} are treated the same way through the Gr\"onwall inequality, thus we only expose the case of \eqref{eqxgron} :\\
Notice that $\tilde x_t^\epsilon(x,\omega)=\tilde x_{t-\epsilon}^\epsilon(\tilde x_{\epsilon}^\epsilon(x,\omega),T(\omega))$. Thus \eqref{eqxgron} can be bounded by
\begin{align}
|\tilde x_t^\epsilon(\omega)-\tilde x_t^\epsilon(x,T(\omega))|&\leq |\tilde x_{t-\epsilon}^\epsilon(\tilde x_{\epsilon}^\epsilon(x,\omega),T(\omega))-\tilde x_{t-\epsilon}^\epsilon(x,T(\omega))|\label{eqzzz}\\
&+|\tilde x_t^\epsilon(x,T(\omega))-\tilde x_{t-\epsilon}^\epsilon(x,T(\omega))|\label{eqzz2}.
\end{align} 
Applying the Gr\"onwall lemma on \eqref{eqzzz}, 
\begin{align*}
\|\tilde x_{t-\epsilon}^\epsilon(\tilde x_{\epsilon}^\epsilon(x,\omega),T(\omega))-\tilde x_{t-\epsilon}^\epsilon(x,T(\omega))\|_\infty\leq K|x-\tilde x_{\epsilon}^\epsilon(x,\omega)|
=\mathcal O(\varepsilon).
\end{align*}
The equation \eqref{eqzz2} is $\mathcal O(\varepsilon)$ since $\left|\frac{d\widetilde x_t^\epsilon}{dt}\right|\le \Vert F\Vert_\infty+\Vert\tau\Vert_\infty\Vert \overline f\Vert_\infty$.
Since $\lim_{\epsilon\rightarrow 0}\mathfrak a_\varepsilon=0$, we conclude that \eqref{eqetzweim} is satisfied and Zweim\"uller's Theorem \cite[Theorem 1.1]{zweimuller} applies ensuring the strong convergence in distribution with respect to $\mu$ of $((\tilde e_{\epsilon n_{t/\epsilon}}^\epsilon(x,\cdot))_{t\geq 0})_{\epsilon >0}$ towards $(y_{t/\overline \tau}(x,\cdot))_{t\geq 0}$.
\end{proof}

\begin{prop}
For all $x\in\mathbb R^d$, the family of processes $
((\mathfrak a_\epsilon/\epsilon)E_t^\epsilon(x,\cdot))_{t\in [0,S]})_\epsilon$ strongly converges in distribution
with respect to $\nu$
to $(Y_t(x,\cdot):=y_{t/\overline\tau}(x,\cdot))_{t\in[0;S]}$, which is given by
\begin{align*}
Y_t(x,\cdot)&:=
 y_{t/\overline\tau}(x,\cdot)= v_{t/\overline\tau} (x,\cdot)+\int_0^{t/\overline\tau} \overline\tau  D \overline f (w_s(x))_s y_s(x,\cdot) \, ds\\
 &=V_{t}(x,\cdot) +\int_0^{t}   D \overline f (W_s(x))_s Y_{s}(x,\cdot)\,  ds\, ,
\end{align*}
with $V_t:=v_{t/\overline\tau}$.
\end{prop}

\begin{proof}
Let $x\in\mathbb R^d$ and $\mathbb{P}$ be a probability measure over $\mathcal{M}$ absolutely continuous with respect to $\nu$. Let us write $\Pi:\mathcal M\rightarrow M$ for the canonical projection given by $\Pi(\omega,s)=\omega$. Since $\Pi_*\mathbb P$ is
absolutely continuous with respect to $\mu$, 
it follows from Lemma~\ref{lemmup} that
$((\mathfrak a_\epsilon/\epsilon)\tilde e^\epsilon_{\epsilon n_{t/\epsilon}(\Pi(\cdot))}(x,\Pi(\cdot))_{t\in[0,S]})_{\epsilon}$ converges in distribution with respect to $\mathbb{P}$ to $(Y_{t}(x,\cdot))_{t\in[0,S]}$. 
Furthermore,~\eqref{eqxtild} and~\eqref{eqwtilde} imply 
 that $\sup_{t\in [0;S],\omega\in\mathcal M}\left|E_t^\epsilon(x,\omega)-\tilde e^\epsilon_{\epsilon n_{t/\epsilon}(\Pi(\omega))}(x,\Pi(\omega))\right|=\mathcal O(\epsilon)$. Since $\lim_{\epsilon\rightarrow 0}\mathfrak a_\epsilon=0$, we conclude that
$((\mathfrak a_\epsilon/\epsilon)(E^\epsilon(x,\Pi(\cdot)))_{t\in[0,S]})_{\epsilon}$ converges in distribution with respect to $\mathbb{P}$ to $(Y_{t}(x,\cdot))_{t\in[0,S]}$.  
\end{proof}

\subsection*{Aknowledgement} This article partly results from my PhD thesis under the supervision of Françoise Pène who initiated this work and provided precious advice and editorial support all along.

\bibliography{biblio_con}
\bibliographystyle{plain}
\end{document}